\documentclass[a4paper,11pt]{article}
\pdfoutput=1

\usepackage[a4paper,tmargin=3truecm,bmargin=3truecm,rmargin=2.5truecm,
lmargin=2.5truecm,twoside,verbose=true]{geometry}
\usepackage{cancel,graphicx}

\usepackage{amsmath,amssymb}
\usepackage[amsmath, hyperref, thmmarks]{ntheorem}
\usepackage[all]{xypic}
\usepackage[pdftex]{hyperref}
\usepackage[english]{babel}

\numberwithin{equation}{section}

\theoremsymbol{}
\theorembodyfont{\slshape}
\theoremheaderfont{\normalfont\bfseries}
\theoremseparator{}
\newtheorem{Theorem}{Theorem}[section]
\newtheorem{theorem}[Theorem]{Theorem}

\newtheorem{proposition}[Theorem]{Proposition}

\newtheorem{lemma}[Theorem]{Lemma}

\newtheorem{corollary}[Theorem]{Corollary}
\theorembodyfont{\upshape}
\theoremsymbol{\ensuremath{\blacklozenge}}

\newtheorem{example}[Theorem]{Example}

\newtheorem{remark}[Theorem]{Remark}

\newtheorem{definition}[Theorem]{Definition}
\theoremstyle{nonumberplain}
\theoremheaderfont{\scshape}
\theorembodyfont{\normalfont}
\theoremsymbol{\ensuremath{\blacksquare}}

\newtheorem{proof}{Proof}
\qedsymbol{\ensuremath{_\blacksquare}}
\theoremclass{LaTeX}

\newcommand\caB{{\mathcal B}}
\newcommand\caC{{\mathcal C}}
\newcommand\caD{{\mathcal D}}
\newcommand\caE{{\mathcal E}}
\newcommand\caF{{\mathcal F}}

\newcommand\caL{{\mathcal L}}

\newcommand\caM{{\mathcal M}}
\newcommand\caS{{\mathcal S}}

\newcommand\caU{{\mathcal U}}
\newcommand\caZ{{\mathcal Z}}

\newcommand\gone{{ \mathchoice {1\mskip-4mu\mathrm{l} } {1\mskip-4mu\mathrm{l} }{1\mskip-4.5mu\mathrm{l} } {1\mskip-5mu\mathrm{l}} }}
\newcommand\gR{{\mathbb R}}

\newcommand\gC{{\mathbb C}}

\newcommand\gS{{\mathbb S}}
\newcommand\gN{{\mathbb N}}
\newcommand\gZ{{\mathbb Z}}

\newcommand\algA{{\mathbf A}}

\newcommand\algB{{\mathbf B}}
\newcommand\algM{{\mathbf M}}
\newcommand\algN{{\mathbf N}}

\newcommand\ehH{\mathcal H}
\newcommand\bS{{\boldsymbol S}}
\newcommand\bG{{\boldsymbol G}}
\newcommand\bH{{\boldsymbol H}}

\newcommand\bb{{\text{\textup{b}}}}

\newcommand\kB{{\mathfrak B}}

\newcommand\kg{{\mathfrak g}}

\newcommand\ks{{\mathfrak s}}

\newcommand\kL{{\mathfrak L}}

\newcommand\kM{{\mathfrak M}}
\newcommand\kMtop{{\mathfrak M}_{\rm topo}}
\newcommand\kMstab{{\mathfrak M}_{\rm stab}}
\newcommand\eps{{\varepsilon}}
\newcommand\ad{{\text{\textup{ad}}}}

\newcommand\fois{\mathord{\cdot}}
\DeclareMathOperator{\tr}{Tr} 

\DeclareMathOperator{\Aut}{\mathsf{Aut}}
\DeclareMathOperator{\Dom}{\mathsf{Dom}}

\newcommand\dd{{\text{\textup{d}}}}

\newcommand\norm{\mathord{\parallel}}

\newcommand\defin{\bf}

\title{Multipliers of Hilbert algebras and deformation quantization\footnote{Work supported by the Belgian Interuniversity Attraction Pole (IAP) within the framework ``Dynamics, Geometry and Statistical Physics'' (DYGEST).}}
\author{Axel de Goursac}
\date{}
\begin{document}

\maketitle

\vspace*{-1cm}
\begin{center}
\textit{Charg\'e de Recherche au F.R.S.-FNRS\\ IRMP, Universit\'e Catholique de Louvain,\\ Chemin du Cyclotron, 2, B-1348 Louvain-la-Neuve, Belgium,}\\
\textit{and Max Planck Institut f\"ur Mathematik,\\ Vivatsgasse 7, D-53111 Bonn, Germany\\
  e-mail: \texttt{axelmg@melix.net}}\\
\end{center}%

\vskip 2cm

\begin{abstract}
In this paper, we introduce the notion of multiplier of a Hilbert algebra. The space of bounded multipliers is a semifinite von Neumann algebra isomorphic to the left von Neumann algebra of the Hilbert algebra, as expected. However, in the unbounded setting, the space of multipliers has the structure of a *-algebra with nice properties concerning commutant and affiliation: it is a pre-GW*-algebra. And this correspondence between Hilbert algebras and its multipliers is functorial. Then, we can endow the Hilbert algebra with a nice topology constructed from unbounded multipliers. As we can see from the theory developed here, multipliers should be an important tool for the study of unbounded operator algebras.

We also formalize the remark that examples of non-formal deformation quantizations give rise to Hilbert algebras, by defining the concept of Hilbert deformation quantization (HDQ) and studying these deformations as well as their bounded and unbounded multipliers in a general way. Then, we reformulate the notion of covariance of a star-product in this framework of HDQ and multipliers, and we call it a symmetry of the HDQ. By using the multiplier topology of a symmetry, we are able to produce various functional spaces attached to the deformation quantization, like the generalization of Schwartz space, Sobolev spaces, Gracia-Bondia-Varilly spaces. Moreover, the non-formal star-exponential of the symmetry can be defined in full generality and has nice relations with these functional spaces. We apply this formalism to the Moyal-Weyl deformation quantization and to the deformation quantization of K\"ahlerian Lie groups with negative curvature.
\end{abstract}

\vskip 1cm


\vspace{3mm}

\pagebreak

\tableofcontents
\vskip 1cm

\section{Introduction}

\subsection{Motivations}

The theory of operator algebras (see \cite{Takesaki:2000,Blackadar:2006}), C*-algebras and von Neumann algebras, was initiated by some mathematicians in the context of quantum mechanics during the 40ies and in particular by J. von Neumann and his celebrated bicommutant theorem \cite{vonNeumann:1931}. Briefly speaking, a von Neumann algebra is a *-algebra of bounded operators on a Hilbert space with identity operator and weakly closed. Later, Dixmier introduced Hilbert algebras for the classification of semifinite von Neumann algebras and to show the commutant theorem \cite{Dixmier:1981}. A Hilbert algebra, endowed with a product, an involution and a scalar product, gives rise to a semifinite von Neumann algebra (its left or right von Neumann algebra) and any semifinite von Neumann algebra can be obtained in this way (up to isomorphism). This structure was then extended to the purely infinite case (see \cite{Takesaki:2000}). The von Neumann algebras have nowadays applications in numerous domains of mathematics and physics, like knot theory, representation theory, noncommutative geometry, logics, probability, quantum field theory, statistical mechanics...

On another side, multipliers were introduced for C*-algebras in the 60ies \cite{Busby:1968} and then extended to various contexts as Fr\'echet algebras. This tool showed its importance for the theory of C*-algebras and in particular for locally compact quantum groups. A natural question relies on what multipliers of a Hilbert algebra should be. In a not so surprisingly way, such multipliers form a von Neumann algebra isomorphic to the left or right von Neumann algebra of the given Hilbert algebra. But this is not the end point concerning this structure...
\medskip

Indeed, operators appearing naturally in quantum mechanics are often unbounded as position and momentum operators. Then, mathematicians started to study *-algebras of unbounded operators with a common dense subdomain \cite{Powers:1971,Lassner:1972} called later O*-algebras (see \cite{Schmudgen:1990,Antoine:2002} for a review on the subject). Properties of von Neumann algebras, generalized in appropriate ways to unbounded operator algebras, lead to the notions of EW*-algebras (symmetric O*-algebra $\algM$ with bounded part $\overline{\algM_\bb}$ von Neumann)\cite{Dixon:1971} and GW*-algebras (O*-algebra equal to is bicommutant, with commutant stabilizing the domain, and one additional topological constraint on the domain) \cite{Inoue:1986}.

Then, the question concerning multipliers of Hilbert algebras can be asked in the unbounded context. Can we define a notion of unbounded multiplier of a Hilbert algebra? Does this notion fit in the generalizations of von Neumann algebras that are EW*- or GW*-algebras? We will see in this paper that such multipliers $\kM(\algA)$ of a Hilbert algebra $\algA$ exist but they do not form an EW*-algebra in general, even not a GW*-algebra, but only a pre-GW*-algebra, relaxing the topological constraint on the domain. Moreover, one can associate to $\kM(\algA)$ a GW*-algebra $\kMtop(\algA)$ giving rise to an interesting topology on the domain $\algA$. As it is the case for (bounded) von Neumann algebras, such unbounded multipliers of Hilbert algebras should probably be crucial tools for the study of unbounded operator algebras.

\bigskip

Deformation quantization yields interesting examples of operator algebras related to Poisson geometry and it was also introduced in the context of quantum mechanics \cite{Bayen:1978}. A deformation quantization of Poisson manifold $M$ is the data of an associative product $\star_\theta$ on $\caC^\infty(M)$ (or one of its dense subspace) depending on a deformation parameter $\theta$, constructed from the Poisson structure, and such that it corresponds to the commutative pointwise product for $\theta=0$.

On one side, one can consider formal deformation quantization, where the star-product $f_1\star_\theta f_2=\sum_{k=0}^\infty C_k(f_1,f_2)\theta^k \in\caC^\infty(M)[[\theta]]$ ($f_i\in\caC^\infty(M)$) is a formal power series in the deformation parameter $\theta$. Such deformation quantizations were intensively studied, and definitely classified in \cite{Kontsevich:2003}. For example, there is only one class of deformation on $M=\gR^{2n}$ given by the Moyal product
\begin{equation}
f_1\star_\theta f_2(x)=\sum_{k=0}^\infty \left(\frac{-i\theta}{2}\right)^k\omega(\partial_x,\partial_y)^k f_1(x)f_2(y)_{|y=x}\label{eq-intro-moyal}
\end{equation}
for $f_i\in\caC^\infty(\gR^{2n})$ and $\omega$ a symplectic structure on $\gR^{2n}$. In the more general case where $G=M$ is a Lie group, such a deformation then corresponds to a Drinfeld twist that permits to deform also algebras on which $G$ acts \cite{Drinfeld:1989}.

In a non-formal point of view, more interesting for functional analysis, and where the deformation parameter $\theta$ has now a real value, there is no classification of deformation quantizations, and there are actually only few available examples. For the Abelian group $G=\gR^{2n}$, the star-product \eqref{eq-intro-moyal} acquires a non-formal meaning in the Moyal-Weyl formula
\begin{equation}
f_1\star_\theta f_2(x)=\frac{1}{(\pi\theta)^{2n}}\int f_1(y)f_2(z) e^{-\frac{2i}{\theta}(\omega(y,z)+\omega(z,x)+\omega(x,y))}\dd y\dd z\label{eq-intro-moyalweyl}
\end{equation}
for $f_i\in\caS(\gR^{2n})$, and for which \eqref{eq-intro-moyal} is the asymptotic expansion in $\theta$ close to 0. Rieffel then used this Abelian symmetry and this deformation in order to build a non-formal twist, also called universal deformation formula (UDF), that deforms continuously the C*-algebras on which $\gR^{2n}$ is acting \cite{Rieffel:1993}. The star-exponential associated to the star-product \eqref{eq-intro-moyalweyl} has been well-defined in a non-formal way and produces various applications in harmonic analysis \cite{Bayen:1982,Cahen:1984,Cahen:1985}. Deformation quantizations of Abelian symmetries were also extended to the complex case \cite{Omori:2000,Garay:2013gya}, to the case of supergroups $\gR^{2n|m}$ \cite{Bieliavsky:2010su} and to Abelian $p$-adic groups \cite{Gayral:2014ad}.
\medskip

In order to construct deformation quantizations of non-Abelian groups, Bieliavsky developed a retract method also based on symmetries \cite{Bieliavsky:2002,Bieliavsky:2001os,Bieliavsky:2010kg}. Starting from a non-formal star-product $\star_\theta^1$ on $M$, which is $G_1$-strongly invariant and $G_2$-covariant, if the shared symmetry $H$ contained in both $G_1$ and $G_2$ is ``sufficiently large'', the retract method then constructs another non-formal star-product $\star_\theta^2$ on $M$ that is $G_2$-strongly invariant, as well as the explicit non-formal intertwiner $U_\theta$ between $\star_\theta^1$ and $\star_\theta^2$. This method was successfully applied to find non-formal deformation quantizations $\star_\theta^2$ of negatively curved K\"ahlerian Lie groups, also called normal $j$-groups in \cite{Pyatetskii-Shapiro:1969}, starting from the Moyal-Weyl product $\star_\theta^1$ \cite{Bieliavsky:2001os,Bieliavsky:2010kg}. An associated pseudodifferential calculus and a UDF were also built in \cite{Bieliavsky:2010kg}, and the corresponding star-exponential was exhibited in \cite{Bieliavsky:2013sk} with applications in harmonic analysis of these Lie groups.

It turns out \cite{Rieffel:1993} that the space $L^2(\gR^{2n})$ endowed with the Moyal-Weyl product $\star_\theta^1$, the complex conjugation and its standard scalar product is a Hilbert algebra. In the same way, $L^2(M)$ with the product $\star_\theta^2$ in the case of normal $j$-groups is \cite{Bieliavsky:2010kg} also a Hilbert algebra, while the intertwiner $U_\theta$ can be viewed as a unitary *-homomorphism. Therefore this motivates the introduction of the notion of Hilbert deformation quantization (HDQ) in this paper, just formalizing these observations. As a consequence, bounded and unbounded multipliers can play a role for HDQ. Moreover, the symmetries of HDQ will allow defining in full generality various functional spaces adapted to these deformations. For example, such functional spaces in the case of $M=\gR^{2n}$ were very useful to define spectral triple on the Moyal-Weyl deformation quantization \cite{Gayral:2003dm}. Furthermore, the non-formal star-exponential can also be defined in a general way by using multipliers and has interesting relations with these functional spaces. All these applications stress that the context of Hilbert algebras and its bounded and unbounded multipliers is very adapted to non-formal deformation quantization, as well as symmetries. Indeed, this crucial concept of symmetry, already producing deformations by retract method and by Drinfeld twists, gives rise now directly to the topology of functional spaces adapted to these deformations.

Eventually, notice that a non-formal $SL(2,\gR)$-strongly invariant deformation quantization was recently exhibited \cite{Bieliavsky:2008mv}, and one can see that it gives also rise to a HDQ. Then, the formalism developed in this paper can be applied for this HDQ and for the symmetry $SL(2,\gR)$, as for other more general symmetries and deformations. Explicit computation of the non-formal star-exponential associated to this symmetry $SL(2,\gR)$ can be found in \cite{Bieliavsky:2013sl2} as well as its link with multipliers.

\subsection{Outline of the paper}

In section \ref{subsec-hilbalg}, we recall the well-known setting of Hilbert algebras - algebras with involution and scalar product - introduced by Dixmier, and their links with semifinite von Neumann algebras. Then, we present the framework of unbounded operator algebras in section \ref{subsec-ostar}. An O*-algebra on the dense domain $\caD$ is an algebra of unbounded operators whose domains contain $\caD$, which stabilize $\caD$, as well as their adjoints. Such an O*-algebra induces on $\caD$ its graphic topology. Bounded and unbounded commutants of O*-algebras are also recalled in section \ref{subsec-gw}, which give rise to the notion of (pre-) GW*-algebra, an extension of von Neumann algebras to the unbounded setting.
\medskip

In section \ref{subsec-multdef}, we start by defining bounded $\kM_\bb(\algA)$ and unbounded $\kM(\algA)$ multipliers of a given Hilbert algebra $\algA$, inspired directly by the definition of double-centralizers of C*-algebras. We show in section \ref{subsec-multb} that bounded multipliers form a von Neumann algebra isomorphic to the left or right von Neumann algebra, as expected. For the space of unbounded multipliers $\kM(\algA)$, we prove in section \ref{subsec-mult} that it has a structure of O*-algebra, we characterize its bounded commutant and its unbounded bicommutant, and we arrive to the fact that it is a pre-GW*-algebra (non-closed in general). It turns out as seen in section \ref{subsec-multmorph} that this multiplier construction defines a covariant functor between Hilbert algebras with isomorphisms as arrows and the category of von Neumann algebras or pre-GW*-algebras with spatial isomorphisms as arrows.

We consider in section \ref{subsec-multtopo} natural topologies on the Hilbert algebra $\algA$ associated to multipliers. It turns out that the graphic topology of $\kM(\algA)$ is not interesting here, but we can associate to $\kM(\algA)$ another O*-algebra $\kMtop(\algA)$, whose graphic topology $\tau_\kM$ on $\algA$ is called the multiplier topology; and this topology will have interesting applications to deformation quantization. We also endow $\kM(\algA)$ with a natural locally convex complete topology, and we look in section \ref{subsec-multcommut} at some particular examples as multipliers of unital or commutative Hilbert algebras.
\medskip

In section \ref{subsec-hdq}, we first define in a general way the concept of Hilbert deformation quantization (HDQ), which is just a non-formal deformation quantization with a Hilbert algebra structure, and we give some basic properties concerning its multipliers. For a given HDQ $\algA_\theta$, we then introduce in section \ref{subsec-symm} the crucial notion of symmetry $\kg$, which expresses the covariance of the deformation quantization in the language of HDQ. We show that a symmetry $\kg$ of a HDQ $\algA_\theta$ forms a subalgebra of multipliers whose multiplier topology generates a Fr\'echet algebra, Hilbert subalgebra of $\algA_\theta$, denoted by $\bS(\algA_\theta,\kg)$ in analogy of the Schwartz space. Moreover, any representation of $\algA_\theta$ with some natural conditions, i.e. quantization map associated to the deformation quantization, can be extended to unbounded multipliers $\kM(\bS(\algA_\theta,\kg))$.

In this operator setting of HDQ, we also define in a general way the star-exponential of a symmetry $\kg$ in section \ref{subsec-starexp}, which is a unitary multiplier satisfying the Baker-Campbell-Hausdorff property. Furthermore, other interesting functional spaces, $\bG^{kl}(\algA_\theta,\kg)$ generalizing topological spaces of Gracia-Bondia-Varilly, and $\bH^k(\algA_\theta,\kg)$ generalizing Sobolev spaces, are also associated to any symmetry $\kg$ in section \ref{subsec-funcsymm}. In the case of an invariant symmetry (the deformation $\star_\theta$ is strongly invariant), all these functional spaces can be characterized as smooth vectors for group actions in section \ref{subsec-invsymm}.
\medskip

In section \ref{subsec-moyal}, we look at the well-known example of Moyal-Weyl deformation quantization that defines a HDQ $\algA_\theta$, and for which the Weyl quantization is a representation. The translation group induces an invariant symmetry $\kg$ of this HDQ $\algA_\theta$. For this basic example, we find for $\bS(\algA_\theta,\kg)$ the usual Schwartz space $\caS(\gR^{2n})$, for $\bG^{kl}(\algA_\theta,\kg)$ the topological spaces introduced by Gracia-Bondia-Varilly with the matrix basis (see section \ref{subsec-gbv}), for $\bH^k(\algA_\theta,\kg)$ the usual Sobolev spaces $H^k(\gR^{2n})$, and for the star-exponential the usual exponential. And we show in section \ref{subsec-moyaldistr} that the unbounded multipliers $\kM(\bS(\algA_\theta,\kg))$ corresponds to the usual $\star_\theta$-multipliers of the Fr\'echet algebra $\caS(\gR^{2n})$. To illustrate the unital case, we also mention in section \ref{subsec-clifford} the infinite-dimensional Clifford algebras, seen as limit of HDQs.

Finally, we consider in section \ref{subsec-normal} the non-formal deformation quantization of K\"ahlerian Lie groups with negative curvature, which is also a HDQ $\algA_\theta$ with a representation. The bounded symmetric domains associated to these groups, and on which the HDQ lives, possess an interesting transvection group. Such a transvection group indeed induces an invariant symmetry of the HDQ $\algA_\theta$ (bigger than the one induced by the K\"ahlerian group itself). A theorem in section \ref{subsec-normalsymm} shows that $\bS(\algA_\theta,\kg)$ is identical to the modified Schwartz space introduced by Bieliavsky-Gayral for this deformation. The star-exponential for this symmetry can also be obtained explicitly by using computations of a previous paper. We conclude by expressing in a explicit way the new functional spaces $\bH^k(\algA_\theta,\kg)$ generated by this symmetry and adapted to this star-product.

\section{Operator Algebra framework}

To fix notations and conventions and to be self-contained, we recall the basics of Hilbert algebras theory (see \cite{Dixmier:1981} and references therein, where all the proofs are given) and the theory of O*-algebras and GW*-algebras (see \cite{Schmudgen:1990,Antoine:2002}) that will be useful in this paper.

\subsection{Hilbert algebras}
\label{subsec-hilbalg}

We just recall here what are Hilbert algebras and give some of their fundamental properties.
\begin{definition}
\label{def-hilbertalg}
Let $\algA$ be an algebra over $\gC$ with an involution and a scalar product\footnote{We choose the convention that the scalar product $\langle -,-\rangle$ is left antilinear and right linear.}. We say that $\algA$ is a {\defin Hilbert algebra} if
\begin{enumerate}
\item for any $x,y\in\algA$, $\langle y^*,x^*\rangle=\langle x,y\rangle$,
\item for any $x,y,z\in\algA$, $\langle xy,z\rangle=\langle y,x^*z\rangle$,
\item for any $x\in\algA$, the map $\lambda_x:y\in\algA\mapsto xy$ is continuous,
\item the set $\{xy,\ x,y\in\algA\}$ is dense in $\algA$.
\end{enumerate}
We note $\ehH_\algA$ the Hilbert space that is the completion of $\algA$ for the norm $\norm x\norm:=\sqrt{\langle x,x\rangle}$ associated to the scalar product.
\end{definition}
For $x\in\algA$, it turns out that $\rho_x:y\in\algA\mapsto yx$ is also continuous. The maps $\lambda,\rho$ extend to a *-algebra morphism $\lambda:\algA\to\caB(\ehH_\algA)$ and a *-algebra antimorphism $\rho:\algA\to\caB(\ehH_\algA)$, where $\caB(\ehH)$ denotes the bounded operators on the Hilbert space $\ehH$. We note $\lambda(\algA)$ (resp. $\rho(\algA)$) the weak closure of the image of $\algA$ by $\lambda$ (resp. $\rho$) and they are called {\defin left} (resp. {\defin right}) {\defin von Neumann algebras} of the Hilbert algebra $\algA$. They satisfy $\lambda(\algA)'=\rho(\algA)$.

The involution extends continuously to a continuous operator on $\ehH_\algA$. Moreover, we can show that $\norm x^*\norm=\norm x\norm$ for all $x\in\ehH_\algA$, but the norm do not satisfy in general the C*-property. The space of Hilbert-Schmidt operators $\caL^2(\ehH)$ of a Hilbert space $\ehH$ is an example of Hilbert algebra.

\begin{definition}
An element $x\in\ehH_\algA$ is called {\defin bounded} if there exists $\lambda_x\in\caB(\ehH_\algA)$ such that $\forall y\in\algA$, $\lambda_x(y)=\rho_y(x)$, or equivalently if there exists $\rho_x\in\caB(\ehH_\algA)$ such that $\forall y\in\algA$, $\rho_x(y)=\lambda_y(x)$. We note $\algA_\bb$ the set of all bounded elements in $\ehH_\algA$, also called the {\defin fulfillment} of $\algA$.

A Hilbert algebra is called {\defin full} if it contains all the bounded elements, i.e. $\algA=\algA_\bb$.
\end{definition}
For any Hilbert algebra $\algA$, $\algA_\bb$ is also a Hilbert algebra (with same Hilbert completion) and by using $\lambda,\rho$, $\algA_\bb$ is included in $\lambda(\algA)$ and $\rho(\algA)$.

\begin{theorem}
Let $\algA$ be a Hilbert algebra. For $S\in\lambda(\algA)^+$ (resp. $T\in\rho(\algA)^+$), we define
\begin{equation*}
\tau_\lambda(S):=\langle x,x\rangle,\qquad \tau_\rho(T):=\langle y,y\rangle,
\end{equation*}
if $S^{\frac12}=\lambda_x$ (resp. $T^{\frac12}=\rho_y$) for some $x$ (resp. $y$) bounded in $\ehH_\algA$; and $\tau_\lambda(S):=+\infty$ (resp. $\tau_\rho(T):=+\infty$) otherwise. Then, $\tau_\lambda$ (resp. $\tau_\rho$) is a faithful semifinite normal trace on $\lambda(\algA)^+$ (resp. $\rho(\algA)^+$). And $\tau_\lambda(\lambda_x^*\lambda_y)=\langle x,y\rangle$ for any $x,y\in\algA_\bb$.
\end{theorem}
The traces $\tau_\lambda$ and $\tau_\rho$ are called the {\defin natural traces} on $\lambda(\algA)^+$ and $\rho(\algA)^+$. It turns out that any von Neumann algebra with semifinite faithful normal trace is isomorphic to the left or right von Neumann algebra of a Hilbert algebra.
\begin{corollary}
\label{cor-5}
Let $\algA$ be a Hilbert algebra. If $x\in\algA_\bb$ bounded satisfies $\rho_y(x)=0$ for all $y\in\algA$, then $x=0$.
\end{corollary}
\begin{proof}
Suppose that $\rho_y(x)=0$ for any $y\in\algA$. Then, it can be extended for any $y\in\algA_\bb$. Take $y=x^*$, we have $\lambda_x\lambda_{x^*}=0$ and we can apply the trace: $0=\tau_\lambda(\lambda_x\lambda_{x^*})=\langle x^*,x^*\rangle$. Then $x=0$.
\end{proof}

\begin{proposition}
Let $\algA_1$ and $\algA_2$ be Hilbert algebras. On the algebraic direct sum $\algA_1\oplus\algA_2$, one can define the following natural structures:
\begin{itemize}
\item Vector space: $(a_1,a_2)+\lambda(b_1,b_2):=(a_1+\lambda b_1,a_2+\lambda b_2)$,
\item Product: $(a_1,a_2)(b_1,b_2):=(a_1 b_1,a_2b_2)$,
\item Involution: $(a_1,a_2)^*:=(a_1^*,a_2^*)$,
\item Scalar product: $\langle (a_1,a_2),(b_1,b_2)\rangle:=\langle a_1,b_1\rangle+\langle a_2,b_2\rangle$,
\end{itemize}
for $a_i,b_i\in\algA_i$. Then, $\algA_1\oplus\algA_2$ is a Hilbert algebra and its completion is $\ehH_{\algA_1}\widehat\oplus\ehH_{\algA_2}$.
\end{proposition}

\begin{proposition}
Let $\algA$ and $\algB$ be two Hilbert algebras. Then the natural structure of the algebraic tensor product are $\algA\otimes\algB$: $\forall a_i\in\algA$, $\forall b_i\in\algB$,
\begin{equation*}
(a_1\otimes b_1)(a_2\otimes b_2):=(a_1a_2)\otimes(b_1b_2),\quad (a\otimes b)^*:= a^*\otimes b^*,\quad \langle a_1\otimes b_1,a_2\otimes b_2\rangle:=\langle a_1,a_2\rangle\langle b_1,b_2\rangle.
\end{equation*}
Then, $\algA\otimes\algB$ is a Hilbert algebra and $\ehH_{\algA\otimes\algB}=\ehH_\algA\widehat\otimes\ehH_\algB$ (completed tensor product with respect to the scalar product).
\end{proposition}

\begin{definition}
\label{def-centerhilb}
Let $Z_\algA=\{z\in\algA_\bb,\, \lambda_z=\rho_z\}$ be the {\defin bounded center} relative to the Hilbert algebra $\algA$.
\end{definition}
This is a subspace of $\ehH_\algA$ depending only on $\algA_\bb$, not on the choice of the Hilbert algebra $\algA$ \cite{Dixmier:1981}.

\subsection{O*-algebras}
\label{subsec-ostar}

In general, unbounded operators cannot be composed and they do not form an algebra. However, if the operators considered have a common dense subdomain and if they stabilize this subdomain, then composition is well-defined and one obtains an algebra. By adding the same conditions for the adjoint in order to get a *-algebra, we arrive to the following definitions.

Let $\ehH$ be a Hilbert space and $\caD$ be a dense subspace of $\ehH$. Let also
\begin{equation*}
\caL^+(\caD):=\{T:\caD\to\caD\ \text{ linear with }\caD\subset\Dom(T^*)\text{ and }T^*(\caD)\subset\caD\},
\end{equation*}
where $\Dom(T)$ denotes the domain of the operator $T$. It is a *-algebra of (closable) unbounded operators with the usual composition of operators and the involution given by the restriction of the adjoint to $\caD$: $T^+:=(T^*)_{|\caD}$. The identity on $\caD$ will be denoted by $\gone$ and belongs to $\caL^+(\caD)$.
\begin{definition}
An {\defin O*-algebra} of domain $\caD$ is a *-subalgebra of $\caL^+(\caD)$ containing $\gone$.
\end{definition}
If $\algM$ is an O*-algebra of domain $\caD$, we denote by $\algM_\bb$ its bounded part, i.e. the bounded operators $T$ of $\caB(\ehH)$ such that $T_{|\caD}\in\algM$.
\begin{definition}
\label{def-graphictop}
Let $\algM$ be an O*-algebra of domain $\caD$. The {\defin graphic topology} of $\caD$ is a locally convex one associated to the seminorms $\norm x\norm_T:=\norm T(x)\norm$, for $T\in\algM$, $x\in\caD$ and $\norm x\norm=\sqrt{\langle x,x\rangle}$.
\end{definition}
This topology denoted by $\tau_\algM$ is stronger than the one induced by the Hilbert space $\ehH$. We note $\tau_+$ the graphic topology on $\caD$ associated to the O*-algebra $\caL^+(\caD)$. Let $\caL^+(\caD,\tau_\algM)$ be the space of operators $T\in\caL^+(\caD)$ such that $T$ and $T^+$ are continuous for the topology $\tau_\algM$. Then, it is an O*-algebra containing $\algM$, and the graphic topology on $\caD$ associated to it coincide with $\tau_\algM$. Note also that if $\tau_\algM=\tau_+$, which is the case if $(\caD,\tau_\algM)$ is a Fr\'echet space, then $\caL^+(\caD,\tau_\algM)=\caL^+(\caD)$.

\begin{definition}
\label{def-closed}
An O*-algebra $\algM$ on the domain $\caD$ is said to be {\defin closed} if $(\caD,\tau_\algM)$ is complete.
\end{definition}
If $\algM$ is not closed, then $\tilde\caD:=\bigcap_{T\in\algM}\Dom(\overline T)$ is the completion of $\caD$, and $\tilde\algM:=\{(\overline T)_{|\tilde\caD},\ T\in\algM\}$ is a closed O*-algebra, called the {\defin closure} of $\algM$. Be care of the fact that $\tilde\caD\subsetneq\bigcap_{T\in\algM}\Dom( T^*)$ in general.

\subsection{GW*-algebras}
\label{subsec-gw}

\begin{definition}
\label{def-commutb}
The {\defin weak bounded commutant} of an O*-algebra $\algM$ is defined as
\begin{equation*}
\algM_w':=\{T\in\caB(\ehH)\ \forall x,y\in\caD,\ \forall S\in\algM,\ \langle y,TS(x)\rangle=\langle S^+(y),T(x)\rangle\}.
\end{equation*}
It is a weakly closed *-invariant subspace, but not an algebra in general.
\end{definition}
There exist other notions of bounded (or unbounded) commutants than the ones presented here but we don't need them in the following.

\begin{definition}
\label{def-affil}
Let $T$ be a closed operator on $\ehH$ and $\algN$ a von Neumann algebra on $\ehH$. We recall that $T$ is said {\defin affiliated} with $\algN$ if $T$ commutes with all operators in $\algN'$, i.e. if $\forall S\in\algN'$, $ST\subset TS$, that is $\forall x\in\Dom(T)$, $S(x)\in\Dom(T)$ and $TS(x)=S(T(x))$.
\end{definition}

\begin{proposition}
\label{prop-condgw}
Let $\algM$ be an O*-algebra on $\caD$. Then $\algM_w'(\caD)\subset \tilde\caD$ if and only if $\overline T$ is affiliated with $(\algM'_w)'$ for any $T\in\algM$. If this is satisfied, then $\algM'_w$ is a von Neumann algebra.
\end{proposition}

The usual weak, strong and strong* topologies can also be defined on $\caL^+(\caD)$. They are respectively defined by the seminorms $\langle x,T(y)\rangle$, $\norm T(x)\norm$, $\norm T(x)\norm+\norm T^+(x)\norm$, for $T\in\caL^+(\caD)$ and $x,y\in\caD$, and denoted by $\tau_w$, $\tau_s$ and $\tau_{s^*}$. We denote by $[\algM]^{s*}$ the closure of $\algM\subset\caL^+(\caD)$ in $\caL^+(\caD)$ for the topology $\tau_{s*}$.
\begin{definition}
\label{def-commut}
The {\defin weak unbounded commutant} of an O*-algebra $\algM$ is defined as
\begin{equation*}
\algM_c':=\{T\in\caL^+(\caD)\ \forall x,y\in\caD,\ \forall S\in\algM,\ \langle S^+(y),T(x)\rangle=\langle T^+(y),S(x)\rangle\}.
\end{equation*}
It is an O*-algebra on $\caD$. The {\defin unbounded bicommutant} of $\algM$ is
\begin{equation*}
\algM''_{wc}:=(\algM'_w)'_c
\end{equation*}
and it is an O*-algebra $\tau_{s*}$-closed (in $\caL^+(\caD)$), of weak bounded commutant $(\algM''_{wc})'_w=\algM_w'$.
\end{definition}

\begin{remark}
\label{rmk-wc}
If $\algM$ is an O*-algebra such that $\algM_w'$ is a von Neumann algebra, then $\algM''_{wc}=[(\algM'_w)'_{|\caD}]^{s^*}$.
\end{remark}

\begin{definition}
\label{def-gw}
Let $\algM$ be an O*-algebra. If $\algM'_w(\caD)\subset\tilde\caD$ and $\algM=\algM''_{wc}$, $\algM$ is called a {\defin pre-GW*-algebra}. If moreover $\algM$ is closed ($\tilde\caD=\caD$), $\algM$ is called a {\defin GW*-algebra}.
\end{definition}

\begin{proposition}
\label{prop-gwext}
If $\algM$ is a pre-GW*-algebra on $\caD$, then its closure $\tilde\algM$ is a GW*-algebra on $\tilde\caD$ and
\begin{equation*}
\algM_w'=(\tilde\algM)_w',\qquad \algM_{wc}''=[(\algM_w')']^{s*}_{\caL^+(\caD)}.
\end{equation*}
\end{proposition}

\begin{proposition}
\label{prop-gw}
Let $\algM$ be a closed O*-algebra. Then, $\algM$ is a GW*-algebra if and only if there exists a von Neumann algebra $\algN$ on $\ehH$ such that $\algN'(\caD)\subset\caD$ and $\algM=[\algN_{|\caD}]^{s^*}$. In this case, $(\algM'_c)'_c=\algM$ and $\algN=(\algM'_w)'$.
\end{proposition}

\begin{example}
\label{ex-l+}
Let $\caD$ be a dense subspace of $\ehH$ such that $\caL^+(\caD)$ is closed. Then, $\caL^+(\caD)$ is a GW*-algebra on $\caD$, its weak bounded commutant is $\caL^+(\caD)'_w=\gC\gone$ so that its bounded bicommutant is $(\caL^+(\caD)'_w)'=\caB(\ehH)$. However, note that $(\caL^+(\caD)'_w)'\neq \caL^+(\caD)_\bb$.
\end{example}

\section{Theory of multipliers}
\label{sec-mult}

\subsection{Definition}
\label{subsec-multdef}

We consider here a Hilbert algebra $\algA$ and we use the notations introduced in section \ref{subsec-hilbalg}. Then, we can adapt the definition of multiplier of a C*-algebra or a Fr\'echet algebra to this context of Hilbert algebras in two different ways: one within the context of unbounded operators leaving $\algA$ invariant and the other within the context of bounded operators.

\begin{definition}
\label{def-mult}
We define an {\defin unbounded multiplier} (also called simply a multiplier) of the Hilbert algebra $\algA$ to be a pair $T=(L,R)$ of operators $L,R\in\caL^+(\algA)$ such that
\begin{equation*}
\forall x,y\in\algA\quad:\quad xL(y)=R(x)y,
\end{equation*}
where $\caL^+(\algA)$ has been given in section \ref{subsec-gw}. Let us denote by $\kM(\algA)$ the set of all (unbounded) multipliers of $\algA$. This is a subset of $\caL^+(\algA\oplus\algA)$.
\end{definition}

\begin{definition}
\label{def-multb}
On another side, we define a {\defin bounded multiplier} of $\algA$ to be a pair $T=(L,R)\in\caB(\ehH_\algA)$ such that
\begin{equation*}
\forall x,y\in\algA\quad:\quad \lambda_x L(y)=\rho_y R(x).
\end{equation*}
Note that we didn't write $xL(y)=R(x)y$ as before because $L(y)$ and $R(x)$ do not belong to $\algA$ a priori. We denote by $\kM_\bb(\algA)$ the set of all bounded multipliers of $\algA$. This is a subset of $\caB(\ehH_\algA\oplus\ehH_\algA)$.
\end{definition}

\begin{lemma}
\label{lem-bounded}
If $x\in\ehH_\algA$ is bounded, i.e. $x\in\algA_\bb$, then for any $T=(L,R)\in\kM_\bb(\algA)$, $L(x)$ and $R(x)$ are in $\algA_\bb$.
\end{lemma}
\begin{proof}
First if $x\in\algA$, Definition \ref{def-multb} gives that $L(x)$ and $R(x)$ are bounded and that $\lambda_{R(x)}=\lambda_xL$, $\rho_{L(x)}=\rho_xR$. Then we have $\rho_{L(y)}(x)=\rho_yR(x)$ for any $x,y\in\algA$, and this extends to any $x\in\algA_\bb$ by density. This means that $\lambda_xL(y)=\rho_{L(y)}(x)=\rho_yR(x)$ for $x\in\algA_\bb$ and $y\in\algA$. Thus $R(x)$ is bounded and $\lambda_{R(x)}=\lambda_xL$. We can do the same thing for $L(x)$.
\end{proof}
This means that $\algA_\bb$ is a *-algebra stable under bounded multipliers. Actually, it turns out that if $\algB$ is dense Hilbert subalgebra of $\algA$, $\kM_\bb(\algB)=\kM_\bb(\algA)$, so that the bounded multipliers depend only of the data of the algebra of bounded elements.

\begin{definition}
\label{def-module}
\begin{itemize}
\item A {\defin left multiplier} (resp. {\defin right multiplier}) of $\algA$ is an operator $L\in\caL^+(\algA)$ (resp. $R\in\caL^+(\algA)$) satisfying
\begin{equation*}
\forall x,y\in\algA\quad:\quad L(xy)=L(x)y\qquad (\text{resp. } R(xy)=xR(y)\ ).
\end{equation*}
We note $L\in\kM_L(\algA)$ (resp. $R\in\kM_R(\algA)$).
\item A {\defin left bounded multiplier} (resp. {\defin right bounded multiplier}) of $\algA$ is an operator $L\in\caB(\ehH_\algA)$ (resp. $R\in\caB(\ehH_\algA)$) satisfying
\begin{equation*}
\forall x,y\in\algA\quad:\quad L(xy)=\rho_yL(x)\qquad (\text{resp. } R(xy)=\lambda_xR(y)\ ).
\end{equation*}
We note $L\in\kM_{L,\bb}(\algA)$ (resp. $R\in\kM_{R,\bb}(\algA)$).
\end{itemize}
\end{definition}
So, (bounded or unbounded) left (resp. right) multipliers are right (resp. left) $\algA$-module homomorphisms. Let us see their relation with multipliers of Definitions \ref{def-mult} and \ref{def-multb}. Denote by $J:x\mapsto x^*\in\caB(\ehH_\algA)$ the involution of $\algA$ extended to $\ehH_\algA$.
\begin{proposition}
\label{prop-moduleb}
Let $L,R\in\caB(\ehH_\algA)$. We have the equivalence between
\begin{enumerate}
\item $(L,R)$ is a bounded multiplier in $\kM_\bb(\algA)$.
\item $L$ is in $\kM_{L,\bb}(\algA)$ and $R=JL^*J$.
\item $R$ is in $\kM_{R,\bb}(\algA)$ and $L=JR^*J$.
\end{enumerate}
\end{proposition}
\begin{proof}
\begin{itemize}
\item $1\Rightarrow 2$: let $(L,R)\in\kM_\bb(\algA)$. Then $\forall x,y,z\in\algA$, $\lambda_x L(yz)=\rho_{yz}R(x)=\rho_z\lambda_xL(y)=\lambda_x\rho_z L(y)$. Corollary \ref{cor-5} and Lemma \ref{lem-bounded} imply that $L\in\kM_{L,\bb}(\algA)$. Next for $x,y,z\in\algA$,
\begin{multline*}
\langle R(x^*z)^*,y\rangle=\langle (\lambda_{x^*}R(z))^*,y\rangle=\langle \rho_x R(z)^*,y\rangle=\langle x,\rho_y R(z)\rangle\\
=\langle x,\lambda_zL(y)\rangle=\langle z^*x,L(y)\rangle=\langle L^*(z^*x),y\rangle,
\end{multline*}
by using standard properties $(\lambda_x(z))^*=\rho_{x^*}(z^*)$ and $\langle \rho_x(z),y\rangle=\langle x,\rho_y(z^*)\rangle$, for $x,y\in\algA$ and $z\in\algA_\bb$. This gives $L^*(z^*x)=R(x^*z)^*$. By using the fact that $L^*(xy)=\rho_yL^*(x)$ thanks to
\begin{equation*}
\langle L^*(xy),z\rangle=\langle xy,L(z)\rangle=\langle x,\rho_{y^*}L(z)\rangle=\langle x,L(zy^*)\rangle=\langle \rho_yL^*(x),z\rangle,
\end{equation*}
we deduce that $R=JL^*J$. 
\item For the reciproc $2\Rightarrow 1$, we suppose that $L\in\kM_{L,\bb}(\algA)$ and $R=JL^*J$. Then for $x,y,z\in\algA$,
\begin{equation*}
\langle \lambda_{y^*}L^*(x),z\rangle=\langle L^*(x),yz\rangle=\langle x,L(yz)\rangle=\langle x,\rho_zL(y)\rangle=\langle \rho_x(L(y)^*),z\rangle,
\end{equation*}
so $\lambda_{y^*}L^*(x)=\rho_x(L(y)^*)$. As a consequence,
\begin{equation*}
\rho_yR(x)=\rho_y(L^*(x^*)^*)=(\lambda_{y^*}L^*(x^*))^*=(\rho_{x^*}(L(y)^*))^*=\lambda_xL(y),
\end{equation*}
and $(L,R)\in\kM_\bb(\algA)$.
\item We can show the equivalence $1\Leftrightarrow 3$ just as above.
\end{itemize}
\end{proof}

With the same arguments as in the above Proposition, we can show an analogue statement in the unbounded case.
\begin{proposition}
\label{prop-module}
Let $L,R\in\caL^+(\algA)$. We have the equivalence between
\begin{enumerate}
\item $(L,R)$ is a multiplier in $\kM(\algA)$.
\item $L$ is in $\kM_{L}(\algA)$ and $R=JL^+J$.
\item $R$ is in $\kM_{R}(\algA)$ and $L=JR^+J$.
\end{enumerate}
\end{proposition}

\subsection{Structure of bounded multipliers}
\label{subsec-multb}

In this section, we will show that $\kM_\bb(\algA)$ is a von Neumann algebra and that it is isomorphic to the left $\lambda(\algA)$ or the right $\rho(\algA)$ von Neumann algebra associated to the Hilbert algebra $\algA$.

First we endow $\kM_\bb(\algA)$ with the following structure: if $T_i=(L_i,R_i)\in\kM_\bb(\algA)$ and $\mu\in\gC$,
\begin{itemize}
\item vector space: $T_1+\mu T_2:=(L_1+\mu L_2,R_1+\mu R_2)$,
\item product: $T_1 T_2:=(L_1L_2,R_2R_1)$,
\item adjoint: $T^*:=(L^*,R^*)$,
\item norm: $\norm T\norm:=\norm L\norm=\norm R\norm$, since $R=JL^*J$ (see Proposition \ref{prop-moduleb}) and $J$ conserves the norm.
\end{itemize}

\begin{lemma}
\label{lem-ximultb}
For $x$ in $\algA_\bb$, we define $\Xi_x:=(\lambda_x,\rho_x)$. Then $\Xi_x\in\kM_\bb(\algA)$. The map $\Xi:\algA_\bb\to\kM_\bb(\algA)$ is an injective *-algebra homomorphism. Moreover, for any $T=(L,R)$ in $\kM_\bb(\algA)$) and $x\in\algA_\bb$, we have
\begin{equation*}
T\Xi_x=\Xi_{L(x)},\qquad \Xi_xT=\Xi_{R(x)}.
\end{equation*}
\end{lemma}
\begin{proof}
Let us show this Lemma in the bounded case. For $x\in\algA_\bb$ and $y,z\in\algA$, $\lambda_y\lambda_x(z)=\lambda_y\rho_z(x)=\rho_z\lambda_y(x)=\rho_z\rho_x(y)$.
This means that $\Xi_x:=(\lambda_x,\rho_x)$ lies in $\kM_\bb(\algA)$. Due to properties of $\lambda$ and $\rho$, $\Xi:x\mapsto \Xi_x$ is a *-algebra homomorphism. Moreover if $\Xi_x=0$, it is immediate by Corollary \ref{cor-5} to see that $x=0$, so that this homomorphism is injective. By using Lemma \ref{lem-bounded} as well as Definition \ref{def-multb} and Proposition \ref{prop-moduleb}, we find that $T\Xi_x=\Xi_{L(x)}$ and $\Xi_xT=\Xi_{R(x)}$. The arguments are similar in the unbounded case.
\end{proof}

\begin{lemma}
\label{lem-multvN}
Endowed with the above structure, $\kM_\bb(\algA)$ is a von Neumann algebra.
\end{lemma}
\begin{proof}
Indeed, $T\in\kM_\bb(\algA)\mapsto \norm T\norm$ is a norm of algebra:
\begin{equation*}
\norm T_1+T_2\norm\leq\max(\norm L_1\norm+\norm L_2\norm, \norm R_1\norm+\norm R_2\norm)\leq \max(\norm L_1\norm,\norm R_1\norm)+\max(\norm L_2\norm,\norm R_2\norm)=\norm T_1\norm+\norm T_2\norm.
\end{equation*}
In the same way, $\norm T_1T_2\norm\leq\norm T_1\norm\,\norm T_2\norm$. Moreover, $(T_1T_2)^*=T_2^* T_1^*$ and
\begin{equation*}
\norm T^*T\norm =\max(\norm L^*L\norm,\norm RR^*\norm)=\max(\norm L\norm^2,\norm R\norm^2)=\max(\norm L\norm,\norm R\norm)^2=\norm T\norm^2.
\end{equation*}
Let us now show that $\kM_\bb(\algA)$ is closed in $\caB(\ehH_\algA)\oplus\caB(\ehH_\algA)$ for the weak topology. Consider a sequence $(T_n)=(L_n,R_n)$ of $\kM_\bb(\algA)$ that converges to $T=(L,R)\in\caB(\ehH_\algA)\oplus\caB(\ehH_\algA)$, we want to show that $T\in\kM_\bb(\algA)$. Indeed,
\begin{multline*}
|\langle z,\lambda_xL(y)-\rho_yR(x)\rangle|\leq |\langle z,\lambda_xL(y)-\lambda_xL_n(y)\rangle|+|\langle z,\rho_yR_n(x)-\rho_yR(x)\rangle|\\
\leq \norm \lambda_x\norm |\langle z,(L-L_n)(y)\rangle|+ \norm \rho_y\norm |\langle z,(R_n-R)(x)\rangle|\to 0
\end{multline*}
when $n\to\infty$. So $T=(L,R)\in\kM_\bb(\algA)$ and $\kM_\bb(\algA)$ is a von Neumann algebra, with unit $\gone:=(\text{id}_{\ehH_\algA},\text{id}_{\ehH_\algA})$.
\end{proof}

\begin{theorem}
\label{thm-caract}
The bounded multiplier algebra of a Hilbert algebra $\algA$ satisfies
\begin{equation*}
\kM_\bb(\algA)=\Xi(\algA_\bb):=\{\Xi_x,\ x\in\algA_\bb\}''.
\end{equation*}
\end{theorem}
\begin{proof}
Due to Lemma \ref{lem-ximultb}, $\Xi:\algA_\bb\to\kM_\bb(\algA)$ is an injective algebra morphism and we see that $\{\Xi_x,\ x\in\algA_\bb\}\simeq \algA_\bb$ is a *-ideal of the von Neumann algebra $\kM_\bb(\algA)$, by using also Lemma \ref{lem-multvN}, so its weak closure $\Xi(\algA)=\Xi(\algA_\bb)$ is also a *-ideal. Let $(L_1,R_1)$ be the unit of $\Xi(\algA)$. Then for any $x,y\in\algA$,
\begin{equation*}
L_1(xy)=L_1\lambda_x(y)=\lambda_x(y)=xy,\qquad R_1(yx)=R_1\rho_x(y)=\rho_x(y)=yx.
\end{equation*}
This means that $L_1$ and $R_1$ are the identity on the space $\{xy,\ x,y\in\algA\}$ that is dense in $\algA$ by 4th axiom of Definition \ref{def-hilbertalg}. Therefore, $L_1=R_1=\text{id}_{\ehH_\algA}$. Finally, the unit $(\text{id}_{\ehH_\algA},\text{id}_{\ehH_\algA})$ of $\kM_\bb(\algA)$ is contained in its *-ideal $\Xi(\algA)$, so that $\kM_\bb(\algA)=\Xi(\algA)$.
\end{proof}
As a consequence, we have $\lambda(\algA)=\kM_{L,\bb}(\algA)$ and $\rho(\algA)=\kM_{R,\bb}(\algA)$. We see in particular with this characterization that the multiplier algebra of Hilbert-Schmidt operators is $\kM_\bb(\caL^2(\ehH))\simeq\caB(\ehH)$, due to the identification between bounded operators on $\ehH$ and bounded operators on $\caL^2(\ehH)$.

\begin{proposition}
\label{prop-commutant}
We can describe the commutant of $\kM_\bb(\algA)$ by 
\begin{equation*}
\kM_\bb(\algA)'=\{\begin{pmatrix} R_1& R_2J \\ JR_3& JR_4J \end{pmatrix},\ R_i\in\kM_{R,\bb}(\algA)\},
\end{equation*}
as subspace of $\caB(\ehH_\algA\hat\oplus\ehH_\algA)$.
\end{proposition}
\begin{proof}
We represent a multiplier $(L,R)\in\kM_\bb(\algA)$ by the matrix $T=\begin{pmatrix} L& 0 \\ 0& JLJ \end{pmatrix}$ in $\caB(\ehH_\algA\hat\oplus\ehH_\algA)$. Let $S=\begin{pmatrix} S_1& S_2 \\ S_3& S_4 \end{pmatrix}$ in $\kM_\bb(\algA)'$ also seen as a subspace of $\caB(\ehH_\algA\hat\oplus\ehH_\algA)$. The commutation relation $TS=ST$ can be expressed as: for any $L\in\kM_{L,\bb}(\algA)$,
\begin{equation*}
S_1 L=LS_1,\quad S_2JLJ=LS_2,\quad S_3L=JLJ S_3,\quad S_4JLJ=JLJS_4.
\end{equation*}
By using the usual commutant theorem, this means for example that $S_1$ lies in $\lambda(\algA)'=\rho(\algA)$. By Proposition \ref{prop-module}, $S_1$ is then a left $\algA$-module homomorphism, i.e. $S_1\in \kM_{R,\bb}(\algA)$. We can use the same arguments for $S_2J$, $JS_3$ and $JS_4J$.
\end{proof}

Moreover, we can consider the natural semifinite faithful normal trace on $\kM_\bb(\algA)^+$ given by: for any $T=(L,R)\in\kM_\bb(\algA)^+$,
\begin{equation*}
\tau(T):=\tau_\lambda(L)=\tau_\rho(R)=\langle x,x\rangle,
\end{equation*}
if $L^{\frac12}=\lambda_x$ that is equivalent to $R^{\frac12}=\rho_x$. Otherwise, it is given by $\tau(T):=+\infty$.

Due to the identification with left or right von Neumann algebras (see Theorem \ref{thm-caract}), we have now the following results for the bounded multipliers of a direct sum or a tensor product of Hilbert algebras.
\begin{proposition}
Let $\algA$ and $\algB$ be two Hilbert algebras. Then we have $\kM_\bb(\algA\oplus\algB)\simeq\kM_\bb(\algA)\oplus\kM_\bb(\algB)$.
\end{proposition}

\begin{proposition}
Let $\algA$ and $\algB$ be two Hilbert algebras. Then we have $\kM_\bb(\algA\otimes\algB)\simeq\kM_\bb(\algA)\overline\otimes\kM_\bb(\algB)$, where $\overline\otimes$ denotes the tensor product of von Neumann algebras.
\end{proposition}

\subsection{Structure of unbounded multipliers}
\label{subsec-mult}

In this section, we will show that the space of (unbounded) multipliers $\kM(\algA)$ on a Hilbert algebra $\algA$ possesses a structure of pre-GW*-algebra strongly related to the von Neumann algebra of bounded multipliers $\kM_\bb(\algA)$. As in section \ref{subsec-multb}, we endow $\kM(\algA)$ with basic operations: if $T_i=(L_i,R_i)\in\kM(\algA)$ and $\mu\in\gC$,
\begin{itemize}
\item vector space: $T_1+\mu T_2:=(L_1+\mu L_2,R_1+\mu R_2)$,
\item product: $T_1 T_2:=(L_1L_2,R_2R_1)$,
\item restriction of the adjoint to $\algA$: $T^+:=(L^+,R^+)$.
\end{itemize}

\begin{lemma}
\label{lem-ximult}
For $x\in\algA$, the restriction of $\Xi_x$ to $\algA\oplus\algA$ lies in $\kM(\algA)$. The map $\Xi:\algA\to\kM(\algA)$ is an injective *-algebra homomorphism. Moreover, for any $T=(L,R)$ in $\kM(\algA)$ and $x\in\algA$, we have
\begin{equation*}
T\Xi_x=\Xi_{L(x)},\qquad \Xi_xT=\Xi_{R(x)}.
\end{equation*}
\end{lemma}
\begin{proof}
We can proceed as in the proof of Lemma \ref{lem-ximultb}.
\end{proof}

\begin{proposition}
For the above operations, $\kM(\algA)$ is an O*-algebra of domain $\algA\oplus\algA$, containing $\algA$ as a *-ideal.
\end{proposition}
\begin{proof}
Let us show that $\kM(\algA)$ is a *-subalgebra of the O*-algebra $\caL^+(\algA)\oplus \caL^+(\algA)$ containing the identity. We see easily that it is a vector subspace containing $(\gone_\algA,\gone_\algA)$. Let $T_i=(L_i,R_i)\in\kM(\algA)$. Then $\forall x,y\in\algA$, $x L_1L_2(y)=R_1(x)L_2(y)=R_2R_1(x)y$, so $T_1T_2\in\kM(\algA)$. Moreover if $T=(L,R)\in\kM(\algA)$, then $L^+=JRJ$ and $R^+=JLJ$ due to Proposition \ref{prop-module}, so $T^+=(L^+,R^+)$ satisfies the identity of multipliers, and $\kM(\algA)$ is an O*-algebra. Lemma \ref{lem-ximult} shows that $\algA\simeq \{\Xi_x,\ x\in\algA\}$ is a *-ideal of $\kM(\algA)$.
\end{proof}
Contrary to the bounded multipliers, $\kM(\algA)$ depends strongly on $\algA$ and not only on $\algA_\bb$.

\begin{lemma}
\label{lem-wbcommut}
The weak bounded commutant of $\kM(\algA)$ coincides with the commutant of $\kM_\bb(\algA)$: $\kM(\algA)'_w=\kM_\bb(\algA)'$.
\end{lemma}
\begin{proof}
As in the proof of Proposition \ref{prop-commutant}, we represent operators $T=(L,R)\in\kM(\algA)$ as matrices $T=\begin{pmatrix} L& 0 \\ 0& R^+ \end{pmatrix}$ (with $R^+=JLJ$) in $\caL^+(\algA\oplus\algA)$. Let $S=\begin{pmatrix} S_1& S_2 \\ S_3& S_4 \end{pmatrix}$ be an element of $\kM(\algA)_w'$. The commutation relation of Definition \ref{def-commutb}, i.e. $\forall T=(L,R)\in\kM(\algA)$, $\forall x_i,y_i\in\algA$
\begin{equation*}
\langle (y_1,y_2), ST(x_1,x_2)\rangle= \langle T^+(y_1,y_2),S(x_1,x_2)\rangle,
\end{equation*}
leads to the following system of equations
\begin{align*}
&\langle y_1,S_1L(x_1)\rangle= \langle L^+(y_1),S_1(x_1)\rangle,\qquad \langle y_2,S_3L(x_1)\rangle=\langle JL^+J(y_2),S_3(x_1)\rangle,\\
&\langle y_1,S_2JLJ(x_2)\rangle= \langle L^+(y_1),S_2(x_2),\rangle,\qquad \langle y_2,S_4JLJ(x_2)\rangle=\langle JL^+J(y_2),S_4(x_2).\rangle
\end{align*}
In particular, for $L=\lambda_z$ with $z\in\algA$, we have
\begin{equation*}
\langle y,S_1(zx)\rangle=\langle y,S_1L(x)\rangle=\langle L^+(y),S_1(x)\rangle=\langle z^*y,S_1(x)\rangle=\langle y,\lambda_zS_1(x)\rangle,
\end{equation*}
so that $S_1$ lies in $\kM_{R,\bb}(\algA)$. We can proceed in a similar way for $S_2J$, $JS_3$ and $JS_4J$.

Reciprocally, let $R_1\in\kM_{R,\bb}(\algA)$, there exist $y_n\in\algA$ such that $\rho_{y_n}\to R_1$ for the weak topology (see Theorem \ref{thm-caract}). Then $\forall x,z\in\algA$, $\forall L\in\kM_L(\algA)$
\begin{equation*}
\langle z,R_1L(x)\rangle=\lim_{n\to\infty}\langle z,\rho_{y_n}L(x)\rangle=\lim_{n\to\infty}\langle z,L(xy_n)\rangle=\langle L^+(z),R_1(x)\rangle.
\end{equation*}
Therefore, $\begin{pmatrix} R_1& R_2J \\ JR_3& JR_4J \end{pmatrix}$, for $R_i\in\kM_{R,\bb}(\algA)$, are elements of the commutant $\kM(\algA)'_w$.
\end{proof}

\begin{lemma}
\label{lem-affilmult}
For any multiplier $T=(L,R)\in\kM(\algA)$, its closure, defined by $\overline T:=(\overline L,\overline R)$, is affiliated with the von Neumann algebra $\kM_\bb(\algA)$.
\end{lemma}
\begin{proof}
First, we consider the closure $\overline L$ of  $L\in\kM_{L}(\algA)$. Let $x\in\Dom(\overline L)$, there exist $x_n\in\algA$ such that $x_n\to x$ and $L(x_n)\to \overline L(x)$. Then for any $y\in\algA$, we have $\rho_y(x_n)\to\rho_y(x)$ and $L(\rho_y(x_n))=\rho_yL(x_n)\to \rho_y\overline L(x)$ due to Definition \ref{def-module}. Since $\overline L$ is closed, $\rho_y(x)\in\Dom(\overline L)$ for $x\in \Dom(\overline L)$ and $y\in\algA$, and $\overline L(\rho_y(x))=\rho_y(\overline L(x))$.

Let $R_1\in\kM_{R,\bb}(\algA)$, there exist $y_n\in\algA$ such that $\rho_{y_n}\to R_1$ for the strong topology (see Theorem \ref{thm-caract}). Then for $x\in\Dom(\overline L)$, $\rho_{y_n}(x)\to R_1(x)$ and $\overline L(\rho_{y_n}(x))=\rho_{y_n}\overline L(x)\to R_1(\overline L(x))$. So $\forall x\in\Dom(\overline L)$,
\begin{equation}
R_1(x)\in\Dom(\overline L)\text{ and }\overline L(R_1(x))=R_1\overline L(x).\label{eq-affil}
\end{equation}

Now consider $T\in\kM(\algA)$ and its closure $\overline T=(\overline L,\overline R)$. As $\overline{R^+}=J\overline L J$, $\overline T$ can be represented as $\begin{pmatrix} \overline L& 0 \\ 0& J\overline L J \end{pmatrix}$ in the closed operators on $\ehH_\algA\hat\oplus\ehH_\algA$. By using the expression of the elements $\begin{pmatrix} R_1& R_2J \\ JR_3& JR_4J \end{pmatrix}$ of $\kM_\bb(\algA)'$ (see Proposition \ref{prop-commutant}), we see that $\overline T$ is affiliated with $\kM_\bb(\algA)$ (see Definition \ref{def-affil}) if for any $R_i\in\kM_{R,\bb}(\algA)$,
\begin{equation*}
R_1\overline L\subset \overline LR_1,\quad R_2\overline LJ\subset \overline LR_2J,\quad JR_3\overline L\subset J\overline LR_3,\quad JR_4\overline LJ\subset J\overline LR_4J.
\end{equation*}
And this a consequence of Equation \eqref{eq-affil}.
\end{proof}

\begin{lemma}
\label{lem-wcbicom}
The unbounded bicommutant of $\kM(\algA)$ coincides with the multipliers: $\kM(\algA)''_{wc}=\kM(\algA)$.
\end{lemma}
\begin{proof}
Let $S=\begin{pmatrix} S_1& S_2 \\ S_3& S_4 \end{pmatrix}$ in $\kM(\algA)''_{wc}$, so $S_i\in\caL^+(\algA)$. Definition \ref{def-commut} of the bicommutant implies the following system of equations: for any $R_i\in\kM_{R,\bb}(\algA)$, $\forall x_i,y_i\in\algA$
\begin{align*}
&\langle R_1^*(y_1),S_1(x_1)\rangle+\langle JR_2^*(y_1),S_3(x_1)\rangle=\langle S_1^+(y_1),R_1(x_1)\rangle +\langle S_2^+(y_1),JR_3(x_1)\rangle,\\
&\langle R_1^*(y_1),S_2(x_2)\rangle +\langle JR_2^*(y_1),S_4(x_2)\rangle= \langle S_1^+(y_1),R_2J(x_2)\rangle +\langle S_2^+(y_1),JR_4J(x_2)\rangle,\\
&\langle R_3^*J(y_2),S_1(x_1)\rangle+\langle JR_4^*J(y_2),S_3(x_1)\rangle=\langle S_3^+(y_2),R_1(x_1)\rangle +\langle S_4^+(y_2),JR_3(x_1)\rangle,\\
&\langle R_3^*J(y_2),S_2(x_2)\rangle+\langle JR_4^*J(y_2),S_4(x_2)\rangle=\langle S_3^+(y_2),R_2J(x_2)\rangle +\langle S_4^+(y_2),JR_4J(x_2)\rangle.
\end{align*}
since $\begin{pmatrix} R_1& R_2J \\ JR_3& JR_4J \end{pmatrix}$ are the elements of $\kM(\algA)_w'$ due to Lemma \ref{lem-wbcommut}. By making vanish the appropriate $R_i$, we find the equations
\begin{align*}
&\langle R_1^*(y),S_1(x)\rangle=\langle S_1^+(y),R_1(x)\rangle,\qquad \langle R_4^*(y),JS_4J(x)\rangle=\langle JS_4^+J(y),R_4(x)\rangle,\\
&\langle R_1^*(y),S_2(x)\rangle=0,\qquad \langle JR_4^*J(y),S_3(x)\rangle=0,\qquad \langle R_3^*(y),S_1(x)\rangle=\langle JS_4^+J(y),R_3(x)\rangle,
\end{align*}
for any $R_i\in\kM_{R,\bb}(\algA)$ and $\forall x,y\in\algA$. By a similar argument as in the proof of Lemma \ref{lem-wbcommut}, the first line of equations implies that $S_1$ and $JS_4J$ lie in $\kM_{L}(\algA)$. Due to the fourth axiom of Definition \ref{def-hilbertalg}, the $R_1(y)$ generate a dense subspace of $\ehH_\algA$ so we have $S_2=S_3=0$. Finally, $\langle JS_4^+J(y),R_3(x)\rangle=\langle R_3^*(y), JS_4J(x)\rangle$ and $S_4=JS_1J$. We have thus showed that $S\in\kM(\algA)$. The reciproc $\kM(\algA)\subset\kM(\algA)''_{wc}$ is evident.
\end{proof}

\begin{theorem}
Let $\algA$ be a Hilbert algebra. Then, $\kM(\algA)$ is a pre-GW*-algebra.
\end{theorem}
\begin{proof}
In Lemmas \ref{lem-affilmult} and \ref{lem-wbcommut}, we have showed that $\forall T\in\kM(\algA)$, $\overline T$ is affiliated with $\kM_\bb(\algA)=(\kM(\algA)'_w)'$. Due to Proposition \ref{prop-condgw}, we have $\kM(\algA)'_w(\algA\oplus\algA)\subset \widetilde{\algA\oplus\algA}$. Finally, by using Lemma \ref{lem-wcbicom}, $\kM(\algA)''_{wc}=\kM(\algA)$ and the two conditions of Definition \ref{def-gw} are satisfied so that $\kM(\algA)$ is a pre-GW*-algebra.
\end{proof}
As a consequence of this Theorem and of Proposition \ref{prop-gwext}, from a Hilbert algebra $\algA$ we constructed a GW*-algebra given by the closure of $\kM(\algA)$. We also define the {\defin bounded part of multipliers}: $\kMstab(\algA):=\kM_\bb(\algA)_{|(\algA\oplus\algA)}\cap \kM(\algA)$. These are the bounded multipliers stabilizing $\algA$ as well as their adjoints. But in general, $\overline{\kMstab(\algA)}\neq \kM_\bb(\algA)$.

\subsection{Morphisms and unitaries}
\label{subsec-multmorph}

We show here the functorial property of the bounded and unbounded multipliers, and then we characterize unitary multipliers.

\begin{definition}
\label{def-isomhilb}
Let $\algA$ and $\algB$ be two Hilbert algebras. An {\defin isomorphism} of Hilbert algebras between $\algA$ and $\algB$ is a *-algebra homomorphism $\Phi:\algA\to\algB$ that extends to a unitary map $\Phi:\ehH_\algA\to\ehH_\algB$. If $\algB=\algA$, we say that $\Phi$ is an automorphism of $\algA$ and we note $\Phi\in\Aut(\algA)$.
\end{definition}

\begin{proposition}
\label{prop-morph}
Let $\algA$ and $\algB$ be two Hilbert algebras. Given an isomorphism $\Phi:\algA\to\algB$ of Hilbert algebras, we define
\begin{itemize}
\item the bounded extension $\tilde\Phi:\kM_\bb(\algA)\to\kM_\bb(\algB)$ by
\begin{equation*}
\forall T=(L,R)\in\kM_\bb(\algA)\quad:\quad \tilde\Phi(T):=\big(\Phi\circ L\circ\Phi^{-1},\Phi\circ R\circ\Phi^{-1}\big).
\end{equation*}
Then, $\tilde\Phi$ is a spatial isomorphism of von Neumann algebras.
\item the unbounded extension $\tilde\Phi:\kM(\algA)\to\kM(\algB)$ by
\begin{equation*}
\forall T=(L,R)\in\kM(\algA)\quad:\quad \tilde\Phi(T):=\big(\Phi\circ L\circ\Phi^{-1},\Phi\circ R\circ\Phi^{-1}\big).
\end{equation*}
Then, $\tilde\Phi$ is a spatial isomorphism of pre-GW*-algebras.
\end{itemize}
\end{proposition}
\begin{proof}
Let us show this in the bounded case, the unbounded case is similar. First $\tilde\Phi(T)$ is a multiplier of $\algB$. Indeed for any $x,y\in\algB$,
\begin{equation*}
\lambda_x \Phi L\Phi^{-1}(y)=\Phi(\lambda_{\Phi^{-1}(x)} L\Phi^{-1}(y))=\Phi(\rho_{\Phi^{-1}(y)}R\Phi^{-1}(x))=\rho_y\Phi R\Phi^{-1}(x).
\end{equation*}
It is straightforward to see that $\tilde\Phi$ is a morphism of algebras. It conserves the involutions since $\forall x,y\in\algB$,
\begin{equation*}
\langle\Phi L^*\Phi^{-1}(x),y\rangle=\langle L^*\Phi^{-1}(x),\Phi^{-1}(y)\rangle=\langle x,\Phi L\Phi^{-1}(y)\rangle.
\end{equation*}
\end{proof}
Note that if two Hilbert algebras $\algA$ and $\algB$ are isomorphic and $\Phi$ is the isomorphism, then the natural traces are compatible, i.e.
\begin{equation}
\forall T\in\kM_\bb(\algA)\quad:\quad \tau_\algB(\tilde\Phi(T))=\tau_\algA(T).\label{eq-morphtrace}
\end{equation}

\begin{proposition}
Let $T=(L,R)\in\kM_\bb(\algA)$. Then, $L$ is unitary if and only if $R$ is unitary. Moreover, $L$ stabilizes $\algA$ if and only if $R$ stabilizes $\algA$. We will therefore say that $T$ is a {\defin unitary multiplier} if $T\in\kMstab(\algA)$ is unitary.
\end{proposition}
\begin{proof}
Indeed for $T=(L,R)\in\kM_\bb(\algA)$, $R=JL^*J$ and $J$ is a unitary antilinear map.
\end{proof}
Note that if $T\in\kMstab(\algA)$ is a unitary multiplier and $\Phi:\algA\to\algB$ an isomorphism of Hilbert algebras, then $\tilde\Phi(T)$ is a unitary multiplier in $\kM_\bb(\algB)$ (see Proposition \ref{prop-morph}).

\begin{proposition}
Let $T=(L,R)\in\kMstab(\algA)$ be a unitary multiplier. Then, $U_T:=LR^*=R^*L$ is an automorphism of the Hilbert algebra $\algA$. Moreover, if $T$ is involutive (i.e. $T^2=\gone$), so is $U_T$.
\end{proposition}
\begin{proof}
Indeed, $L$ and $R^*$ are respectively in $\kM_{L,\bb}(\algA)$ and $\kM_{R,\bb}(\algA)$ and they commute together due to Theorem \ref{thm-caract}, and $U_T$ is a unitary operator on $\ehH_\algA$ stabilizing $\algA$. Then, it preserves the involution
\begin{equation*}
U_T(x)^*=(LR^*(x))^*=R^*L(x^*)=U_T(x^*),
\end{equation*}
for $x\in\algA$, by using $L^*(x)=R(x^*)^*$. Moreover, $R^*(x)L(y)=RR^*(x)y=xy$ and
\begin{equation*}
U_T(x)U_T(y)=(LR^*)(x)(LR^*)(y)=LR^*(R^*(x)L(y))=U_T(xy).
\end{equation*}
\end{proof}
If an automorphism $U\in\Aut(\algA)$ is associated to a unitary multiplier $T=(L,R)\in\kM_\bb(\algA)$ as above: $U=U_T$, we say that $U$ is an {\defin inner automorphism}. In this case, the induced spatial isomorphism of $\kM_\bb(\algA)$ has the form
\begin{equation*}
\tilde U_T(T')=(LL'L^*,R^*R'R),
\end{equation*}
for any $T'=(L',R')\in\kM_\bb(\algA)$.

\begin{proposition}
Let $\algA$ be a Hilbert algebra and $P=(P_L,P_R)\in\kM_\bb(\algA)$. Then $P_L$ is a projection if and only if $P_R=JP_LJ$ is a projection. In this case, we note $\algB:=P_LP_R(\algA)$. $\algB$ is a Hilbert algebra dense in $\ehH_\algB:=P_LP_R(\ehH_\algA)$ and
\begin{equation*}
\kM_\bb(\algB)=\{(P_LP_R\,L\,P_LP_R,P_LP_R\,R\,P_LP_R),\quad (L,R)\in\kM_\bb(\algA)\}.
\end{equation*}
\end{proposition}
\begin{proof}
Due to the formula $P_R=JP_L^*J$, we see that $P_L^2=P_L=P_L^*$ if and only if $P_R^2=P_R=P_R^*$. Then, the rest can be showed by using Lemma \ref{lem-ximultb}.
\end{proof}

\subsection{Topology on multipliers}
\label{subsec-multtopo}

The multiplier space $\kM(\algA)$ of an arbitrary Hilbert algebra $\algA$ is an O*-algebra so that we can consider the graphic topology (see Definition \ref{def-graphictop}) on its domain $\algA\oplus\algA$. However, we will define another more interesting topology on $\algA$ by introducing an auxiliary O*-algebra called the derived multiplier algebra.
\begin{definition}
Let $\kMtop(\algA)$ be the $s^*$-closure in $\caL^+(\algA)$ of the vector space generated by $LR'$ (for $L\in\kM_L(\algA)$ and $R'\in\kM_R(\algA)$), or equivalently the $s^*$-closure of the algebra generated by $\kM_L(\algA)\oplus\kM_R(\algA)$. This is an O*-algebra of domain $\algA$ called the {\defin derived multiplier algebra}. Then, the {\defin multiplier topology} $\tau_\kM$ on $\algA$ is defined as the graphic topology associated to the O*-algebra $\kMtop(\algA)$. $\algA$ is said to be {\defin mult-closed} if it is complete for the multiplier topology.
\end{definition}
The seminorms of the multiplier topology are $\norm x\norm_T:=\norm T(x)\norm$, for $T\in\kMtop(\algA)$ and $x\in\algA$.

Let us recall that $Z_\algA$ denotes the bounded center (see Definition \ref{def-centerhilb}), and we define the {\defin multiplier center} as $\caZ_\bb(\algA):=\kM_{L,\bb}(\algA)\cap\kM_{R,\bb}(\algA)$.
\begin{lemma}
\label{lem-densder}
The bounded center $Z_\algA$ is strongly dense in $\caZ_\bb(\algA)$.
\end{lemma}
\begin{proof}
First, for any $z\in Z_\algA$, $\lambda_z=\rho_z\in\caZ_\bb(\algA)$. Moreover $\forall T\in\caZ_\bb(\algA)$, $\lambda_{T(z)}=T\lambda_z=T\rho_z=\rho_{T(z)}$ by using Lemma \ref{lem-ximultb}. So $T(z)\in Z_\algA$ and $Z_\algA$ is a *-ideal of $\caZ_\bb(\algA)$. Let $T_1$ denote the unit of the strong closure of $Z_\algA$. As in Theorem \ref{thm-caract}, due to the density of $\{xy,\, x,y\in\algA\}$ in $\ehH_\algA$, we deduce that $T_1=\text{id}_{\ehH_\algA}$ and finally that $\caZ_\bb(\algA)$ is the strong closure of $Z_\algA$.
\end{proof}

\begin{proposition}
\label{prop-commutder}
Let $\algA$ be a Hilbert algebra. Then $\kMtop(\algA)'_w=\caZ_\bb(\algA)$.
\end{proposition}
\begin{proof}
Let $T\in\kMtop(\algA)'_w$. Then, for any $x,y\in\algA$ and $S\in\kMtop(\algA)$, we have $\langle y,TS(x)\rangle=\langle S^+(y),T(x)\rangle$. If we take $S=\lambda_z$ with $z\in\algA$, we obtain that $T$ commutes with $\lambda_z$, so $T\in\kM_{R,\bb}(\algA)$. With $S=\rho_z$, we have $T\in\kM_{L,\bb}(\algA)$, so $T\in\caZ_\bb(\algA)\subset\kMtop(\algA)'_w$.
\end{proof}

\begin{remark}
\label{rmk-bicomder}
Note that $\caZ_\bb(\algA)$ is a von Neumann algebra whose commutant is given by $\kM_{{\rm topo},\bb}(\algA)$, generated by $LR'$ (for $L\in\kM_{L,\bb}(\algA)$ and $R'\in\kM_{R,\bb}(\algA)$). Due to the $s^*$-closure of $\kMtop(\algA)$ and to Remark \ref{rmk-wc}, we see that
\begin{equation*}
\kMtop(\algA)=[\kM_{{\rm topo},\bb}(\algA)_{|\algA}]^{s^*}=[(\kMtop(\algA)'_w)'_{|\algA}]^{s^*}=\kMtop(\algA)''_{wc}.
\end{equation*}
\end{remark}

We say that the Hilbert algebra $\algA$ is {\bf centered} if $Z_\algA\fois\algA\subset\algA$.
\begin{lemma}
\label{lem-centeredder}
Let $\algA$ be a centered Hilbert algebra. Then, for any $T\in\kMtop(\algA)$, the closure $\overline T$ is affiliated with $\kM_{{\rm topo},\bb}(\algA)$.
\end{lemma}
\begin{proof}
We proceed as for Lemma \ref{lem-affilmult}. Let $T\in\kMtop(\algA)$ and $x\in\Dom(\overline T)$, there exists a sequence $x_n\in\algA$ with $x_n\to x$ and $T(x_n)\to \overline T(x)$. For any $z\in Z_\algA$, $z$ is bounded and $\lambda_z(x_n)\to\lambda_z(x)$. Moreover, since $Z_\algA\fois\algA\subset\algA$, we deduce from Remark \ref{rmk-bicomder} that $T$ commutes with $\lambda_z$ and
\begin{equation*}
T(\lambda_z(x_n))=\lambda_z(T(x_n))\to \lambda_z\overline T(x).
\end{equation*}
So $\lambda_z(x)\in\Dom(\overline T)$ and $\overline T(\lambda_z(x))=\lambda_z \overline T(x)$. Due to Lemma \ref{lem-densder}, we can extend this result to all elements of $\caZ_\bb(\algA)=\kM_{{\rm topo},\bb}(\algA)'$ instead of $\lambda_z$.
\end{proof}

\begin{proposition}
\label{prop-gwder}
Let $\algA$ a centered mult-closed Hilbert algebra. Then, $\kMtop(\algA)$ is a GW*-algebra.
\end{proposition}
\begin{proof}
Indeed, Lemma \ref{lem-centeredder} with Proposition \ref{prop-condgw} show that $\kMtop(\algA)'_w$ stabilizes $\algA$. Finally, $\kMtop(\algA)$ coincides with its bicommutant due to Remark \ref{rmk-bicomder}.
\end{proof}

Note that an element $T\in\kMtop(\algA)$ decomposes as $T=\sum_i L_i R'_i$ (limit in the $s^*$-topology) with $L_i\in\kM_L(\algA)$ and $R'_i\in\kM_R(\algA)$, so that $T\in\kMtop(\algA)$ if and only if $JTJ\in\kMtop(\algA)$.

\begin{proposition}
Let us suppose that the multiplier topology is given by the seminorms generated by $LR'$, for a countable number of elements $L\in\kM(\algA)$, $R'\in\kM_R(\algA)$ (without $s^*$-limit). Then, $(\algA,\tau_\kM)$ is a topological *-algebra, i.e. the product is separately continuous on $(\algA,\tau_\kM)$ and the involution is continuous on $(\algA,\tau_\kM)$. Moreover, if $(\algA,\tau_\kM)$ is complete, then it is a Fr\'echet algebra.
\end{proposition}
\begin{proof}
Indeed for $x,y\in\algA$ and $T=\sum_i L_iR'_i\in\kMtop(\algA)$  (we can take the sum finite), we have $\norm x^*\norm_T=\norm x\norm_{JTJ}$ and
\begin{equation*}
\norm xy\norm_T\leq \frac12(\norm \lambda_{L_i(x)}\norm\,\norm y\norm_{R'_i} +\norm \rho_{R'_i(y)}\norm\, \norm x\norm_{L_i}).
\end{equation*}
\end{proof}
Note that the multiplier topology is automatically given by the topology associated to operators $LR'$ (without $s^*$-limit) when this latter topology is Fr\'echet (see below Definition \ref{def-graphictop}).

\begin{proposition}
\label{prop-homeo}
Let $\Phi:\algA\to\algB$ be an isomorphism of Hilbert algebras. Then, $\Phi$ is a homeomorphism for the multiplier topologies of $\algA$ and $\algB$.
\end{proposition}
\begin{proof}
We denote by $\tilde\Phi:\kM(\algA)\to\kM(\algB)$ the unbounded extension of $\Phi$ (see Proposition \ref{prop-morph}). For $x\in\algA$ and $T\in\kMtop(\algA)$, we have
\begin{equation*}
\norm x\norm_T=\norm T(x)\norm=\norm \Phi\circ T\circ\Phi^{-1}(\Phi(x))\norm=\norm \Phi(x)\norm_{\tilde\Phi(T)},
\end{equation*}
and it turns out that $\Phi\circ T\circ\Phi^{-1}\in\kMtop(\algA)$ due to the decomposition $T=\sum_iL_i R_i'$ (limit in the $s^*$-topology).
\end{proof}

\begin{proposition}
\label{prop-conttopo}
Let $(L,R)\in\kM(\algA)$. Then $L$ and $R$ are continuous linear maps $\algA\to\algA$ for the multiplier topology.
\end{proposition}
\begin{proof}
Indeed for $T\in\kMtop(\algA)$ and $x\in\algA$, we have
\begin{equation*}
\norm L(x)\norm_T=\norm T L(x)\norm=\norm x\norm_{TL},\qquad \norm R(x)\norm_T=\norm T R(x)\norm=\norm x\norm_{TR},
\end{equation*}
and as before, $TL$ and $TR$ are in $\kMtop(\algA)$.
\end{proof}

\begin{definition}
Let $\algA$ be a Hilbert algebra. We introduce here the {\defin strong* topology} on the multipliers $\kM(\algA)$. Let $\kB_\kM$ be the space of bounded subsets of $\algA$ for the multiplier topology $\tau_\kM$. The seminorms of the strong* topology are
\begin{equation*}
p_{B,S}(L):=\sup_{x\in B}\norm SL(x)\norm\qquad\text{and}\qquad p_{B,S}(R)=\sup_{x\in B}\norm SR(x)\norm=p_{JB,JSJ}(L^+),
\end{equation*}
for $(L,R)\in\kM(\algA)$, $B\in\kB_\kM$ and $S\in\kMtop(\algA)$.
\end{definition}
\begin{proposition}
\label{prop-lchalg}
Endowed with the strong*-topology, $\kM(\algA)$ is a locally convex Hausdorff *-algebra. Moreover, if $(\algA,\tau_\kM)$ is Fr\'echet, then $\kM(\algA)$ is complete. If $\algA$ is Fr\'echet nuclear, $\kM(\algA)$ is complete nuclear.
\end{proposition}
\begin{proof}
First, let us show that $\kM_L(\algA)$ is a closed subspace of $\caL(\algA,\tau_\kM)$ that is a locally convex Hausdorff space. Let $L\in\caL(\algA,\tau_\kM)$ and $L_n\in\kM_L(\algA)$ converging to $L$ in the strong-topology. For any $x,y\in\algA$, we choose a bounded subset $B\in\kB_\kM$ containing $x$ and $xy$, we have
\begin{equation*}
\norm L(xy)-L(x)y\norm\leq \norm L(xy)-L_n(xy)\norm+\norm L_n(x)y-L(x)y\norm
\end{equation*}
and $\norm L_n(x)y-L(x)y\norm\leq \norm\rho_y\norm\,\norm L_n(x)-L(x)\norm$, so that we obtain $L(xy)=L(x)y$ and $L\in\kM_L(\algA)$. We proceed in the same way for $\kM_R(\algA)$. And $\kM(\algA)\simeq \kM_L(\algA)\cap\kM_R(\algA)$ for the topologies described above. Moreover, the product is separately continuous on $\kM_L(\algA)$ and $\kM_R(\algA)$. Indeed, for $L,L'\in\kM_L(\algA)$,
\begin{equation*}
p_{B,SL}(L')=p_{B,S}(LL')=p_{L'(B),S}(L),
\end{equation*}
with $SL\in\kMtop(\algA)$ and $L'(B)\in\kB_\kM$.

If $\algA$ is Fr\'echet, then $\caL(\algA,\tau_\kM)$ is complete (for the bounded convergence) and $\kM(\algA)$ also for the strong*-topology. If $\algA$ is Fr\'echet nuclear, then $\caL(\algA,\tau_\kM)$ is complete nuclear and $\kM(\algA)$ also.
\end{proof}

\subsection{Special cases}
\label{subsec-multcommut}

Let us see first the case of a unital Hilbert algebra $\algA$. From \cite{Dixmier:1981}, we know that the left von Neumann associated to $\algA$ is finite, so $\kM_\bb(\algA)\simeq\algA_\bb$ is finite.
\begin{proposition}
\label{prop-unit}
Let $\algA$ be a unital Hilbert algebra. Then $\kM(\algA)\simeq \algA$.
\end{proposition}
\begin{proof}
Let indeed $T=(L,R)\in\kM(\algA)$. Then for any $x\in\algA$, $L(x)=L(\gone)x\in\algA$ and $L(\gone)\in\algA$. Moreover, $L(\gone)=R(\gone)$. Therefore, $(L,R)\in\kM(\algA)\mapsto L(\gone)=R(\gone)\in\algA$ is an isomorphism.
\end{proof}
We also see that $\kM(\algA)$ is not closed, so not a GW*-algebra, unless $\algA=\algA_\bb=\ehH_\algA$.
\medskip

The second special case we consider here is commutative Hilbert algebras.
\begin{proposition}
\label{prop-commut1}
Let $\algA$ be a commutative Hilbert algebra. Then $\kM_\bb(\algA)$ is a commutative von Neumann algebra.
\end{proposition}
\begin{proof}
First, for $x,y\in\algA$ we have $\lambda_x(y)=\rho_x(y)$ and this extends by continuity to all $y\in\ehH_\algA$. For $x\in\algA_\bb$ and $y\in\algA$, $\rho_x(y)=\lambda_y(x)=\rho_y(x)=\lambda_x(y)$ and this extends also to all $y\in\ehH_\algA$. So $\algA_\bb=Z_\algA$ and $\kM_\bb(\algA)=\caZ_\bb(\algA)$ due to Theorem \ref{thm-caract} and Lemma \ref{lem-densder}.
\end{proof}

\begin{proposition}
Let $\algA$ be a commutative Hilbert algebra. Then $\kM(\algA)$ is a commutative O*-algebra and $\forall T\in\kM(\algA)$, $T^*=\overline{T^+}$.
\end{proposition}
\begin{proof}
From Lemma \ref{lem-affilmult}, we have that $\overline S,\overline T$ are affiliated with $\kM_\bb(\algA)$ for any $S,T\in\kM(\algA)$. It means that they can be approximated in the $s^*$-topology by sequences $S_n$ and $T_m$ of $\kM_\bb(\algA)$. Since $S_n$ and $T_m$ commute by Proposition \ref{prop-commut1}, we obtain by taking the limits that $S$ and $T$ commute.

The next part is standard and can be found in \cite{Antoine:2002}, but we indicate it here for self-containedness. For $T\in\kM(\algA)$ and $x\in\algA$, we have $\norm T(x)\norm=\norm T^+(x)\norm$, so $\Dom(\overline T)=\Dom(\overline{T^+})$. Let now $x\in\Dom(T^*)$. We consider the polar decomposition $\overline T=U_T |\overline T|$ with $U_T$ a unitary element of $\kM_\bb(\algA)$ (since $\overline T$ is affiliated with). We have for any $y\in\algA$,
\begin{equation*}
\langle y,U_TT^*x\rangle=\langle \overline T U_T^*y,x\rangle=\langle U_T^*\overline Ty,x\rangle=\langle |\overline T|y,x\rangle
\end{equation*}
since $U_T\in\kM_\bb(\algA)\subset\kM_\bb(\algA)'$. Therefore, $x\in\Dom(|\overline T|)=\Dom(\overline T)=\Dom(\overline{T^+})$.
\end{proof}
$\kM(\algA)$ is then called an essentially selfadjoint O*-algebra.

\begin{example}
Let us consider the Schwartz function $\algA=\caS(\gR^{n})$ with commutative pointwise multiplication, complex conjugation and scalar product of $L^2(\gR^{n})$. $\algA$ is a commutative Hilbert algebra and its fulfillment is $\algA_\bb= L^2(\gR^{n})\cap L^\infty(\gR^{n})$. We have obviously $\kM_\bb(\algA)\simeq L^\infty(\gR^{n})$ and $\kM(\algA)\simeq \{T\in\caS'(\gR^{n}),\, \forall f\in\caS(\gR^{2n}),\, Tf=fT\in\caS(\gR^{2n})\}$: the tempered multipliers of $\algA$.
\end{example}

%
%

\section{Multipliers in non-formal deformation quantization}

\subsection{Hilbert deformation quantization}
\label{subsec-hdq}

Owing to the theory of multipliers and to the examples of non-formal deformation quantization we know (see next sections), we introduce the following definition of Hilbert deformation quantization.
\begin{definition}
\label{def-hdq}
Let $M$ be a smooth manifold, and for any $\theta\in\gR$, $\algA_\theta$ be a subspace of complex measurable functions on $M$. Let $\star_\theta$ be an associative product on $\algA_\theta$  for any $\theta\neq 0$ ($\star_0$ be the pointwise product on $\algA_0$), and $\langle-,-\rangle$ be a fixed scalar product.
\begin{itemize}
\item The family $\algA=(\algA_\theta)$ is called a {\defin Hilbert deformation quantization} (HDQ) of $M$ if for any $\theta\in\gR$, $(\algA_\theta,\star_\theta)$ is a Hilbert algebra for the involution given by complex conjugation and for the scalar product $\langle-,-\rangle$, and if $\algA_\theta$ contains the smooth functions with compact support $\caD(M)$ as a dense subset.
\item A  Hilbert deformation quantization $\algA$, with constant fibers $\algA_\theta=\algA$, is called {\defin continuous} if for any $f_1,f_2\in\algA$, the map $\theta\mapsto \norm f_1\star_\theta f_2\norm$ is continuous on $\gR$, for the norm associated to the scalar product.
\item Let $\algA$ and $\algB$ be two Hilbert deformation quantizations of the smooth manifolds $M$ and $N$. An {\defin intertwiner} of $\algA$ and $\algB$ is a family of isomorphisms $U_\theta:\algA_\theta\to\algB_\theta$ of Hilbert algebras (see Definition \ref{def-isomhilb}).
\item A {\defin representation} (also called a quantization map) of a HDQ $\algA$ is a family $\Omega=(\Omega_\theta)$ of isometric *-morphisms $\Omega_\theta:\algA_\theta\to\caL^2(\ehH_\theta)$ with dense range, for $(\ehH_\theta)$ a family of Hilbert spaces.
\end{itemize}
\end{definition}

To any Hilbert deformation quantization (HDQ) $\algA=(\algA_\theta)$, we can associate directly the family of von Neumann algebras given by the bounded multipliers $\kM_\bb(\algA):=(\kM_\bb(\algA_\theta))$ and the family of pre-GW*-algebras given by the unbounded multipliers $\kM(\algA):=(\kM(\algA_\theta))$. Then, we can easily prove the following result.
\begin{proposition}
\label{prop-hdq}
\begin{itemize}
\item If $\algA$ is a $G$-invariant HDQ, then $\kM_\bb(\algA)$ and $\kM(\algA)$ are also stabilized by the natural action of $G$.
\item Let $U:\algA\to\algB$ be an intertwiner between two HDQ. Then, there exist $\tilde T:\kM_\bb(\algA)\to\kM_\bb(\algB)$ a spatial isomorphism of von Neumann algebras and $\tilde T:\kM(\algA)\to\kM(\algB)$ a spatial isomorphism of pre-GW*-algebras.
\end{itemize}
\end{proposition}
\begin{proof}
First, the action of $G$ on the multipliers is $g^*T:=g^*\circ T\circ (g^{-1})^*$, for $g\in G$ and $T\in\kM_\bb(\algA)$ or $T\in\kM(\algA)$. It turns out that $g^*$ is an automorphism of any Hilbert algebra $\algA_\theta$ due to the $G$-invariance, so we get the first result. The second result is just a translation of Proposition \ref{prop-morph} in the framework of HDQ.
\end{proof}

\begin{proposition}
Any representation $\Omega$ of a HDQ $\algA$, $\Omega:\algA\to\caL^2(\ehH)$, extends as an isomorphism of Hilbert algebras $\algA_\bb\simeq\caL^2(\ehH)$ and then as a von Neumann isomorphism $\tilde\Omega:\kM_\bb(\algA)\to\caB(\ehH)$.
\end{proposition}
\begin{proof}
This uses Proposition \ref{prop-hdq} and the standard fact that $\kM_\bb(\caL^2(\ehH))\simeq \caB(\ehH)$ for any Hilbert space $\ehH$. We recall the proof. Consider the map $\Phi: T\in\caB(\ehH)\mapsto (T\circ\fois, \fois\circ T)\in\kM_\bb(\caL^2(\ehH))$. It is an algebra homomorphism compatible with the involutions. But $\norm T\circ\fois\norm_{\caB(\caL^2(\ehH))}=\norm T\norm_{\caB(\ehH)}=\norm \fois\circ T\norm_{\caB(\caL^2(\ehH))}$ so that $\Phi$ is isometric.

Let $(e_k)$ be a Hilbert basis of $\ehH$. Then, due to Parseval theorem, $(\varphi_{kl})$ defined by $\varphi_{kl}(e_m):=\delta_{lm}e_k$, is a Hilbert basis of $\caL^2(\ehH)$: any $S\in\caL^2(\ehH)$ decomposes as $S=\sum_{k,l}S_{kl}\varphi_{kl}$ with $\norm S\norm_2^2=\sum_{k,l}|S_{kl}|^2<\infty$. Any left multiplier $L\in\kM_{L,\bb}(\caL^2(\ehH))$ writes $L(\varphi_{kl})=\sum_{m,n}L_{mnkl}\varphi_{mn}$. Left multiplier condition as well as the identity $\varphi_{kl}\varphi_{mn}=\delta_{lm}\varphi_{kn}$ imply that $L_{mnkl}=L_{mk}\delta_{ln}$. By using the isometry of $\Phi$, we deduce that $L$ coincides with the image of $\Phi$ of the operator $e_k\mapsto \sum_m L_{mk}e_m$ in $\caB(\ehH)$, and $\Phi$ is surjective.
\end{proof}

\begin{proposition}
\label{prop-contfield}
\begin{itemize}
\item Let $\algA$ be a continuous HDQ. Then, the constant family $\algA$ is a lower semicontinuous family of pre-C*-algebras. But in general, its completion doesn't coincide with $\kM_\bb(\algA)$.
\item Let $U:\algA\to\algB$ be an intertwiner between two HDQ and suppose that $\kM_\bb(\algA)$ is a continuous field of C*-algebras. Then, $\kM_\bb(\algB)$ is also a continuous field of C*-algebras.
\end{itemize}
\end{proposition}
\begin{proof}
The first result is due to M. Rieffel. If $\theta\to\norm f_1\star_\theta f_2\norm$ is continuous for any $f_1,f_2\in\algA$, then $\langle f_1, f\star_\theta f_2\rangle$ is also continuous and this implies lower semicontinuity (see \cite{Rieffel:1989cf}). The second result is obvious by using the isomorphism $\tilde U$.
\end{proof}

\subsection{Symmetries of Hilbert deformation quantizations}
\label{subsec-symm}

We assume here that the HDQ is complete, i.e. for any $\theta\in\gR^*$, $\algA_\theta$ is a complete Hilbert algebra. We introduce here the concept of symmetries of HDQ , which will be useful to construct other HDQ. Note that a symmetry of a HDQ $\algA$ of a manifold $M$ does not come in general from a group action on $M$, but these group actions can be interesting examples.

Let $(\algA_\theta,\star_\theta)$ be a complete HDQ of a smooth manifold $M$.

\begin{definition}
\label{def-symhdq}
Let $\kg$ be a countable-dimensional subspace of $\caC^\infty(M)$, stable under complex conjugation, such that left and right $\star_\theta$-multiplication $L_T,R_T$ are defined as unbounded operators with domain containing $\caD(M)$, for any $T\in\kg$. We assume that
\begin{enumerate}
\item $\kg$ is a Lie algebra for the $\star_\theta$-commutator, i.e. $\forall T_1,T_2\in\kg$, $\exists T_3\in\kg$, $[L_{T_1},L_{T_2}]\subset L_{T_3}$ (covariance). Then, $\caU(\kg)$ has a countable PBW basis.
\item the $\star_\theta$-multiplication $L_T,R_T$ are unbounded operators with domain containing $\caD(M)$, for any $T\in\caU(\kg)$,
\item such operators $L_S,R_T$ commute on $\caD(M)$ and act as multipliers on $\caD(M)$,
\item they satisfy $JL_TJ=R_{\overline T}$, where $J$ is the complex conjugation,
\item and they commute with the bounded center $\caZ_\bb(\algA)$.
\end{enumerate}
Under these conditions, $\kg$ is called a {\defin symmetry} of the HDQ $\algA$.
\end{definition}

\begin{remark}
If $\kg$ contains functions such that $L_T$ and $R_T$ are in $\algA_\theta$, for any $T\in\caU(\kg)$, then conditions 2, 3, 4 and 5 are trivial. In the general unbounded case (for next theorem), we have to assume such conditions, but in concrete examples, with explicit expressions of $L_T$ and $R_T$, these conditions will be easy to check.
\end{remark}

\begin{theorem}
\label{thm-hdq1}
Let $\kg$ be a symmetry of the HDQ $\algA$. Then, we set $\algB_{\theta}$ to be the closure of $\caD(M)$ for the seminorms $\norm f\norm_{S,T}:=\norm L_SR_T(f)\norm$, for $S,T\in\caU(\kg)$ and $f\in\caD(M)$.

Then, $\algB_{\theta}$ is a dense Hilbert subalgebra of $\algA_\theta$ and $\algB$ defines therefore a HDQ. Moreover, the above locally convex topology corresponds to the multiplier topology $\tau_\kM$ on $\algB$, and $(\algB,\tau_\kM)$ is a Fr\'echet algebra.
\end{theorem}
\begin{proof}
First $\algB_\theta$ is dense in $\algA_\theta$ because it contains $\caD(M)$. The algebra $\algA_\theta$ is of type $I_\infty$ so it is complete with respect to the scalar product, and hence the product $\star_\theta$ is jointly continuous for this topology. So for any $S,T\in\caU(\kg)$ and $f,g\in\caD(M)$,
\begin{equation*}
\norm L_SR_T(f\star_\theta g)\norm=\norm L_S(f)\star_\theta R_T(g)\norm\leq C\norm L_S(f)\norm\,\norm R_T(g)\norm<\infty,
\end{equation*}
due to condition 3 of Definition \ref{def-symhdq}, and $\algB_\theta$ is an algebra. Due to condition 4 of Definition \ref{def-symhdq}, we have that $\norm L_S R_T(\overline f)\norm=\norm L_{\overline T}R_{\overline S}(f)\norm$ and $\algB_\theta$ is also stable by the complex conjugation.

We showed that $\algB_\theta$ is a dense Hilbert subalgebra of $\algA_\theta$, but it has also a Fr\'echet topology with the seminorms $\norm\fois\norm_{S,T}$. Moreover by definition, restricted to $\algB_\theta$, $L_T\in\kM_L(\algB_\theta)$ and $R_T\in\kM_R(\algB_\theta)$. From condition 5 of Definition \ref{def-symhdq}, we see by Propositions \ref{prop-gwder} and \ref{prop-commutder} that for any $S,T\in\caU(\kg)$, $L_SR_T\in\kMtop(\algB_\theta)$. It means that $\{(L_S)_{|\algB_\theta}(R_T)_{|\algB_\theta},\, S,T\in\caU(\kg)\}$ generates an O*-algebra on $\algB_\theta$ contained in $\kMtop(\algB_\theta)$, and whose graphic topology is Fr\'echet. This topology then coincides with the multiplier topology of $\algB_\theta$ by section \ref{subsec-ostar}.
\end{proof}
In the above notations, we call $\algB$ the {\defin Schwartz HDQ induced} by the symmetry $\kg$ from the complete HDQ $\algA$ and denote it by $\bS(\algA,\kg):=\algB$. Due to Proposition \ref{prop-lchalg}, $\kM(\algB)$ is then a family of locally convex complete *-algebras, and it contains the symmetry $\kg$ and its universal enveloping algebra $\caU(\kg)$.

\begin{definition}
\label{def-symhdqcomp}
Let $\Omega:\algB\to\caL^2(\ehH)$ be a representation. We say that $\Omega$ is {\defin extendable} if
\begin{itemize}
\item For any $T\in\kM(\algB)$, $\tilde\Omega_\theta(T)$ is defined as an unbounded operator on $\ehH_\theta$ with domain containing a common fixed dense subset $\caD$.
\item For any $T,S\in\kM(\algB)$, $\tilde\Omega_\theta(T)\tilde\Omega_\theta(S)=\tilde\Omega_\theta(T\star_\theta S)$.
\item For any $T$, $\tilde\Omega_\theta(\overline T)=\tilde\Omega_\theta(T)^*$, where $\overline T$ means the complex conjugate of $T$.
\end{itemize}
\end{definition}
If the associative product and the representation are given by smooth kernels as it is the case in various examples, these conditions are generally satisfied and easy to prove.

\begin{theorem}
\label{thm-symhdq}
Let $\kg$ be a symmetry of the HDQ $\algA$, $\algB:=\bS(\algA,\kg)$ the Schwartz HDQ induced by $\kg$, and $\Omega:\algA\to\caL^2(\ehH)$ a representation, which is extendable restricted to $\algB$. Then, $\Omega$ extends to a faithful *-representation $\tilde\Omega:\kM(\algB)\to\caL^+(\caD_\algB)$, where $\caD_\algB$ is a dense domain canonically associated to $\algB$.
\end{theorem}
\begin{proof}
First, we note that $\Omega$ is also a representation of the HDQ $\algB$. Due to the proof of Theorem \ref{thm-hdq1}, $\kM(\caU(\kg))$ is an O*-algebra on $\algB$. We then define the domain $\caD_\algB$ to be
\begin{equation*}
\caD_\algB:=\bigcap_{T\in\caU(\kg)} \big(\Dom(\tilde\Omega(T))\cap \Dom(\tilde\Omega(T)^*\big).
\end{equation*}
It is dense in $\ehH$ as containing $\caD$ (see Definition \ref{def-symhdqcomp}). Due to the second condition of Definition \ref{def-symhdqcomp} and to the fact that $\kM(\caU(\kg))$ stabilizes $\algB$ by definition, we see that if $\varphi\in\caD_\algB$ and $f\in\algB$, $\Omega(f)\varphi\in\caD_\algB$. So the maps $\Omega:\algB\to\caL^+(\caD_\algB)$ and $\tilde\Omega:\kM(\caU(\kg))\to\caL^+(\caD_\algB)$ are well-defined as a preliminary step.

Let us extend now these maps to the whole $\kM(\algB)$. For an element $\varphi\in\caD_\algB$, there exists $f_\varphi\in\algA$ such that $\Omega(f_\varphi)=\frac{1}{\norm\varphi\norm^2}|\varphi\rangle\langle\varphi |$, since $\Omega:\algA\to\caL^2(\ehH)$ is an isomorphism. Then, we have formally (if it is not infinite) for any $T=(L_T,R_T)\in\kM(\algB)$,
\begin{align*}
\norm \tilde\Omega(T)\varphi\norm^2_{\ehH}&=\langle \varphi, \tilde\Omega(T)^+\tilde\Omega(T)\varphi\rangle=\frac{1}{\norm\varphi\norm^2}\tr(|\varphi\rangle\langle \varphi, \tilde\Omega(T)^+\tilde\Omega(T)\varphi\rangle \langle\varphi |)= \norm \tilde\Omega(T)\varphi\rangle \langle\varphi |\norm^2_{\caL^2(\ehH)}\\
&=\norm \tilde\Omega(T)\tilde\Omega(f_\varphi)\norm^2_{\caL^2(\ehH)}=\norm L_T(f_\varphi)\norm^2_{\ehH_\algB},\\
\norm \tilde\Omega(T)^+\varphi\norm^2_{\ehH}&=\langle \varphi, \tilde\Omega(T)\tilde\Omega(T)^+\varphi\rangle=\norm R_T(f_\varphi)\norm^2_{\ehH_\algB}.
\end{align*}
A first application of this computation with $T\in\caU(\kg)$ shows that $L_T(f_\varphi)$ and $R_T(f_\varphi)$ are Hilbert-Schmidt, so $f_\varphi\in\algB$. A second application with $T\in\kM(\algB)$ permits to show that $\caD_\algB\subset\Dom(\tilde\Omega(T))\cap \Dom(\tilde\Omega(T)^*)$. Then, with a slight modification, we have
\begin{align*}
&\norm \tilde\Omega(T)\Omega(T')^*\Omega(S)\varphi\norm^2_{\ehH}=\norm L_TR_{T'}L_S(f_\varphi)\norm^2_{\ehH_\algB}<\infty\\
&\norm \tilde\Omega(T)\Omega(T')^*\Omega(S)^*\varphi\norm^2_{\ehH}=\norm L_TR_{T'}R_S(f_\varphi)\norm^2_{\ehH_\algB}<\infty
\end{align*}
for $T,T'\in\caU(\kg)$, $\varphi\in\caD_\algB$ and $S\in\kM(\algB)$, so $\tilde\Omega(S)\in\caL^+(\caD_\algB)$. Due to Definition \ref{def-symhdqcomp}, $\tilde\Omega$ is a *-algebra homomorphism.

Let $S\in\kM(\algB)$ such that $\tilde\Omega(S)=0$. Then, for any $f\in\algB$, we have
\begin{equation*}
\norm L_S(f)\norm_{\ehH_\algB}=\norm \Omega(L_S(f))\norm_{\caL^2(\ehH)}=\norm \tilde\Omega(S)\Omega(f)\norm=0,
\end{equation*}
due to the isometric map $\Omega:B\to\caL^2(\ehH)$. And $S=0$, which shows the injectivity of $\tilde\Omega$.
\end{proof}

\subsection{Star-exponential}
\label{subsec-starexp}

Let $(\algA_\theta,\star_\theta)$ be a complete HDQ of a smooth manifold $M$. Let also $\kg$ be a real symmetry of $\algA$ and we note $\algB:=\bS(\algA,\kg)$ the Schwartz HDQ induced by the symmetry $\kg$.
\begin{definition}
We define the {\defin star-exponential} $E_{\star_\theta}(\frac{i}{\theta}T)$ of an element $T\in\kg$ to be the pair $(e^{\frac{i}{\theta}L_T},e^{\frac{i}{\theta}R_T})$, where the exponential is understood in the sense of continuous functional calculus.
\end{definition}

\begin{theorem}
\label{thm-starexp}
The star-exponential of any element $T\in\kg$ is a unitary multiplier in $\kMstab(\algB)$, and it satisfies the BCH property:
\begin{equation*}
\forall T,T'\in\kg\quad:\quad E_{\star_\theta}(\frac{i}{\theta}T)\star_\theta E_{\star_\theta}(\frac{i}{\theta}T')=E_{\star_\theta}(\frac{i}{\theta}\text{BCH}(T,T'))
\end{equation*}
where $\text{BCH}(T,T')=\log(e^Te^{T'})$ in the Lie algebra $\kg$.
\end{theorem}
\begin{proof}
Since any element $T\in\kg$ is a real smooth function, $L_T$, $R_T$ are essentially selfadjoint operators so that $e^{\frac{i}{\theta}L_T}$ and $e^{\frac{i}{\theta}R_T}$ are unitary operators, and they are respectively equal to the strongly convergent series $\sum_{n=0}^\infty \frac{i^n}{n!\theta^n}(L_T)^n$ and $\sum_{n=0}^\infty \frac{i^n}{n!\theta^n}(R_T)^n$. The BCH property can be obtained by using this approximation by series as in formal deformation quantization (algebraic property).

What remains to be proved is the fact that $E_{\star_\theta}(\frac{i}{\theta}T)$ stabilizes $\algB$. For example, for any $S,T\in\kg$ and $f\in\algB$, we have
\begin{equation*}
\norm L_S e^{\frac{i}{\theta}L_T}f\norm =\norm e^{-\frac{i}{\theta}L_T} L_S e^{\frac{i}{\theta}L_T}f\norm=\norm L_{e^{\frac{-i}{\theta}\ad_{(L_T)}}(L_S)}f\norm<\infty
\end{equation*}
since $e^{\frac{-i}{\theta}\ad_{L_T}}(L_S)$ is the left action of an element of $\kg$. In the same way, we can extend this result to $L_S$ for $S\in\caU(\kg)$ and to $R_S$, which proves that $e^{\frac{i}{\theta}L_T}f\in\algB$ by using Theorem \ref{thm-hdq1}.
\end{proof}

Note that the star-exponential satisfies the following equation
\begin{equation*}
\partial_t E_{\star_\theta}(\frac{it}{\theta}T)=\frac{i}{\theta} (L_T,R_T)\, E_{\star_\theta}(\frac{i}{\theta}T),
\end{equation*}
which could be useful to determine its explicit expression in the concrete examples.

\subsection{Associated functional spaces}
\label{subsec-funcsymm}

Let $(\algA_\theta,\star_\theta)$ be a complete HDQ of a smooth manifold $M$ and let $\kg$ be a real finite-dimensional symmetry of $\algA$. We denote by $(e_i)_{i\in I}$ a basis of $\kg$ (the following theory will be independent of its particular choice).

\begin{definition}
\label{def-sobolev}
From this deformation $\algA$ and its symmetry $\kg$, we define
\begin{itemize}
\item the {\defin Sobolev spaces} $\bH^k(\algA,\kg)$ ($k\in\gN$) {\defin induced} by $\kg$ as the completion of $\caD(M)$ for the norm
\begin{equation*}
\norm f\norm_k:=\sup_{l\leq k,\, i_1,\dots,i_l\in I}\norm (L_{e_{i_1}}-R_{e_{i_1}})(L_{e_{i_2}}-R_{e_{i_2}})\dots (L_{e_{i_l}}-R_{e_{i_l}})f\norm.
\end{equation*}
\item the {\defin GBV spaces}\footnote{GBV stands for Gracia-Bondia-Varilly as we will see that these spaces generalize the spaces defined by GBV for the Moyal deformation by using the matrix basis.} $\bG^{k,l}(\algA,\kg)$ ($k,l\in\gN$) {\defin induced} by $\kg$ as the completion of $\caD(M)$ for the norm
\begin{equation*}
\norm f\norm_{k,l}:=\sup_{p\leq k,\, i_1,\dots,i_p\in I}\, \sup_{q\leq l,\, j_1,\dots,j_q\in I} \norm L_{e_{i_1}}\dots L_{e_{i_p}}R_{e_{j_1}}\dots R_{e_{j_q}} f\norm.
\end{equation*}
\end{itemize}
\end{definition}
To give an intuition, note that $(L_{e_{i}}-R_{e_{i}})$ just corresponds to the inner $\star_\theta$-derivation $[e_i,\fois]_{\star_\theta}$.

\begin{proposition}
For $k\in\gN$, the induced Sobolev space $\bH^k(\algA_\theta,\kg)$ is a dense Hilbert subalgebra of $\algA_\theta$ and it is a Hilbert space for the scalar product
\begin{equation*}
\langle f_1,f_2\rangle_k:= \sum_{l=0}^k \sum_{i_1,\dots,i_l\in I}\langle (L_{e_{i_1}}-R_{e_{i_1}})\dots (L_{e_{i_l}}-R_{e_{i_l}})f_1, (L_{e_{i_1}}-R_{e_{i_1}})\dots (L_{e_{i_l}}-R_{e_{i_l}})f_2 \rangle.
\end{equation*}
\end{proposition}
\begin{proof}
For any $f_1,f_2\in\bH^k(\algA,\kg)$, we can show recursively by using the Leibniz rule for the $\star_\theta$-derivation $(L_{e_{i}}-R_{e_{i}})$ that $f_1\star_\theta f_2\in \bH^k(\algA,\kg)$. Indeed, at the first orders,
\begin{align*}
\norm (L_{e_{i}}-R_{e_{i}})(f_1\star_\theta f_2)\norm &\leq \norm (L_{e_{i}}-R_{e_{i}})f_1\norm\,\norm f_2\norm+\norm f_1\norm\,\norm (L_{e_{i}}-R_{e_{i}})f_2\norm<+\infty,\\
\norm (L_{e_{i_1}}-R_{e_{i_1}})(f_1\star_\theta f_2)\norm &\leq \norm (L_{e_{i_1}}-R_{e_{i_1}})(L_{e_{i_2}}-R_{e_{i_2}})f_1\norm\,\norm f_2\norm +\norm (L_{e_{i_1}}-R_{e_{i_1}})f_1\norm\,\norm (L_{e_{i_2}}-R_{e_{i_2}})f_2\norm\\
& +\norm (L_{e_{i_2}}-R_{e_{i_2}})f_1\norm\,\norm (L_{e_{i_1}}-R_{e_{i_1}})f_2\norm +\norm f_1\norm\,\norm (L_{e_{i_1}}-R_{e_{i_1}})(L_{e_{i_2}}-R_{e_{i_2}})f_2\norm,
\end{align*}
where we also used that the Hilbert algebra $\algA_\theta$ is complete, so the product $\star_\theta$ is jointly continuous for its Hilbert topology. Iteration of such identities yields the result: $f_1\star_\theta f_2\in\bH^k(\algA,\kg)$. For the involution $J$ (corresponding to the complex conjugation), we have
\begin{multline*}
\norm (L_{e_{i_1}}-R_{e_{i_1}})\dots (L_{e_{i_l}}-R_{e_{i_l}})\overline{f}\norm= \norm J(L_{e_{i_1}}-R_{e_{i_1}})JJ\dots JJ(L_{e_{i_l}}-R_{e_{i_l}})Jf\norm\\
= \norm (L_{e_{i_1}}-R_{e_{i_1}})\dots (L_{e_{i_l}}-R_{e_{i_l}})f\norm,
\end{multline*}
because $J$ is unitary and involutive, and $JL_{e_i}J=R_{\overline{e_i}}=R_{e_i}$. So if $f\in\bH^k(\algA,\kg)$, its complex conjugate also belongs to the Sobolev space. It is then easy to see that $\langle -,-\rangle_k$ is a hermitian positive definite scalar product and that it is associated to the topology of $\bH^k(\algA,\kg)$.
\end{proof}

\begin{remark}
We note $H^\infty(\algA,\kg)=\bigcap_{k=0}^\infty H^k(\algA,\kg)$. It is also a dense Hilbert subalgebra of $\algA$, and a Fr\'echet algebra for the projective limit topology.
\end{remark}

\begin{proposition}
For $k,l\in\gN$, the GBV induced spaces $\bG^{k,l}(\algA,\kg)$ satisfy
\begin{equation*}
\bG^{k,p}(\algA_\theta,\kg)\star_\theta\bG^{q,l}(\algA_\theta,\kg)\subset \bG^{k,l}(\algA_\theta,\kg),\qquad J(\bG^{k,l}(\algA_\theta,\kg))=\bG^{l,k}(\algA_\theta,\kg),
\end{equation*}
for any $p,q\in\gN$, and where $J$ is the complex conjugation. As a consequence, $\bG^{k,l}(\algA_\theta,\kg)$ is a subalgebra of $\algA_\theta$ and a Hilbert space for the scalar product
\begin{equation*}
\langle f_1,f_2\rangle_{k,l}:= \sum_{p=0}^k\sum_{q=0}^l \sum_{i_1,\dots,i_p\in I} \sum_{j_1,\dots,j_q\in I}\langle L_{e_{i_1}}\dots L_{e_{i_p}}R_{e_{j_1}}\dots R_{e_{j_q}}f_1, L_{e_{i_1}}\dots L_{e_{i_p}}R_{e_{j_1}}\dots R_{e_{j_q}}f_2 \rangle.
\end{equation*}
Moreover, $\bG^{k,k}(\algA_\theta,\kg)$ is a Hilbert subalgebra of $\algA_\theta$ contained in $\bH^k(\algA_\theta,\kg)$.
\end{proposition}
\begin{proof}
By using the unbounded multiplier property of the symmetry: $L_{e_i}R_{e_j}(f_1\star_\theta f_2)= L_{e_i}(f_1)\star_\theta R_{e_j}(f_2)$, as well as the completeness of the HDQ $\algA_\theta$, we obtain the result concerning the star-product of GBV spaces. Since the complex conjugation transforms a left multiplication into a right one, we have also that $J(\bG^{k,l}(\algA_\theta,\kg))=\bG^{l,k}(\algA_\theta,\kg)$. The rest are easy consequences.
\end{proof}

\begin{remark}
\label{rmk-gbv}
It turns out that the GBV spaces form a decreasing sequence, i.e. $\bG^{k',l'}(\algA,\kg)\subset \bG^{k,l}(\algA,\kg)$ for $k\leq k'$ and $l\leq l'$. Moreover, we can see that $\bS(\algA,\kg)=\bigcap_{k,l\in\gN} \bG^{k,l}(\algA,\kg)$ and that the projective limit topology corresponds to the standard topology defined in Theorem \ref{thm-hdq1}.
\end{remark}

We see in particular that the spaces $H^k(\algA,\kg)$ and $G^{k,k}(\algA,\kg)$ define HDQs, and also non-trivial unbounded multipliers *-algebras.

\begin{lemma}
\label{lem-holom}
For any $k,l,p,q\in\gN$, we have
\begin{equation*}
\bG^{k,l}(\algA,\kg)\star_\theta \kM_\bb(\bS(\algA,\kg))\star_\theta \bG^{p,q}(\algA,\kg)\subset \bG^{k,q}(\algA,\kg).
\end{equation*}
\end{lemma}
\begin{proof}
Let $f_1\in\bG^{k,l}(\algA,\kg)$, $f_2\in\bG^{p,q}(\algA,\kg)$ and $T\in\kM_\bb(\bS(\algA,\kg))$. Then, we have
\begin{multline*}
\norm L_{e_{i_1}}\dots L_{e_{i_p}}R_{e_{j_1}}\dots R_{e_{j_q}}(R_T(f_1)\star_\theta f_2)\norm \leq \norm  L_{e_{i_1}}\dots L_{e_{i_p}} R_T(f_1)\norm\, \norm R_{e_{j_1}}\dots R_{e_{j_q}}f_2\norm\\
\leq \norm R_T\norm\, \norm  L_{e_{i_1}}\dots L_{e_{i_p}} f_1\norm\, \norm R_{e_{j_1}}\dots R_{e_{j_q}}f_2\norm<\infty,
\end{multline*}
by using the fact that $\algA_\theta$ is complete, and that $(R_T,L_T)\in\kM(\bS(\algA,\kg))'_w$ by Lemma \ref{lem-wbcommut} and Proposition \ref{prop-commutant}, so that $L_{e_{i_1}}\dots L_{e_{i_p}}$ and $R_T$ commute.
\end{proof}

\begin{proposition}
\label{prop-holom}
The induced Schwartz space $\bS(\algA,\kg)$ is stable under the holomorphic functional calculus.
\end{proposition}
\begin{proof}
We adapt the argument of \cite{Gayral:2003dm} to this more general case. Let $f\in\bS(\algA,\kg)$ such that $1+f$ is invertible in $\kM_\bb(\algA)$, and we note $1+T$ its inverse. Let us show that $T\in\bS(\algA,\kg)$. Invertibility condition means $f+T+f\star_\theta T=0$ in $\kM_\bb(\bS(\algA,\kg))$ where we identify $f$ with its multiplier $\Xi_f=(\lambda_f,\rho_f)$. We multiply on the right by $f$ and obtain
\begin{equation*}
T=-f-T\star_\theta f=-f+f\star_\theta f +f\star_\theta T\star_\theta f.
\end{equation*}
By using Lemma \ref{lem-holom} and Remark \ref{rmk-gbv}, we have that $f\star_\theta T\star_\theta f\in \bigcap_{k,l\in\gN} \bG^{k,l}(\algA,\kg)=\bS(\algA,\kg)$; so $T\in \bS(\algA,\kg)$, and $\bS(\algA,\kg)$ is stable under holomorphic calculus.
\end{proof}

\subsection{Invariant symmetries}
\label{subsec-invsymm}

To motivate the definition of an invariant symmetry of a HDQ $\algA$, let us consider first a formal deformation quantization $\star_\theta$ of a symplectic manifold $(M,\omega)$, on which the Lie group $G$ is acting in a strong Hamiltonian way, i.e. there exists moment maps $\eta_X\in\caC^\infty(M)$ for any $X\in\kg$ ($\kg$: Lie algebra of $G$), such that the Poisson bracket with respect to $\omega$ gives $\{\eta_X,\eta_Y\}=\eta_{[X,Y]}$.

We recall that $\star_\theta$ is said to be covariant if $[\eta_X,\eta_Y]_{\star_\theta}=-i\theta\eta_{[X,Y]}$, and invariant if $\forall g\in G$  the pullback of the left action leaves the product invariant: $g^*(f_1\star_\theta f_2)=(g^*f_1)\star_\theta (g^*f_2)$. The invariance at the infinitesimal level provides $X^*(f_1\star_\theta f_2)=(X^*f_1)\star_\theta f_2+f_1\star_\theta (X^* f_2)$, for $X\in\kg$ and $X^*=\frac{\dd}{\dd t}_{|t=0} (e^{-tX})^*$ is the fundamental vector field.

Suppose that $M$ is cohomologically trivial, so that any derivation is inner, there exist $\Xi_X\in\caC^\infty(M)$ such that $X^*=\frac{i}{\theta}\,[\Xi_X,\fois]_{\star_\theta}$. But we have
\begin{equation*}
\frac{i}{\theta}\,[\Xi_X-\eta_X,\eta_Y]_{\star_\theta}=X^*(\eta_Y)-\eta_{[X,Y]}=0.
\end{equation*}
If the Lie group $G$ contains ``sufficiently symmetries'', then $\Xi_X=\eta_X$, and $X^*=\frac{i}{\theta}\,[\eta_X,\fois]_{\star_\theta}$, which will be a very useful property. 

Reciprocally, if $X^*=\frac{i}{\theta}\,[\eta_X,\fois]_{\star_\theta}$, then the star-product is obviously covariant and each fundamental vector is a $\star_\theta$-derivation. So, if the Lie group $G$ is connected, the pullback of the left action of $G$ acts by $\star_\theta$-automorphisms, so $\star_\theta$ is also $G$-invariant. This discussion comes from ideas of the retract method of P. Bieliavsky.

We can now introduce the notion of invariant symmetry. Let us consider now non-formal deformation quantization.
\begin{definition}
Let $(\algA_\theta,\star_\theta)$ be a complete HDQ of a smooth manifold $M$. Let $G$ be a connected Lie group acting smoothly and in a strong Hamiltonian way on $M$, with Lie algebra $\kg$. This action is said to be an {\defin invariant symmetry} of the HDQ $\algA_\theta$ if the left action $\kL_g$ of $g\in G$, defined by $\kL_g(f)(x):=f(g^{-1}x)$ ($x\in M$), is an automorphism of the Hilbert algebra $\algA_\theta$ and if
\begin{equation*}
\forall X\in\kg\quad:\quad X^*=\frac{i}{\theta} \,[\eta_X,\fois]_{\star_\theta},
\end{equation*}
where $\eta_X$ denotes the moment map of $X\in\kg$.
\end{definition}
Given such a $G$-invariant symmetry of the HDQ $\algA$, the moment maps $\eta_X$ then form a symmetry of the HDQ $\algA_\theta$. We can consider the star-exponential of the moment maps $E_{\star_\theta}(\frac{i}{\theta}\eta_X)$ (see section \ref{subsec-starexp}) or its integrated version:
\begin{equation*}
\caE_{e^{X}}:=E_{\star_\theta}(\frac{i}{\theta}\eta_X).
\end{equation*}

\begin{proposition}
Let $\algA$ be a complete HDQ with a $G$-invariant symmetry. Then the star-exponential induces an injective homomorphism $\caE:G\to\kM_\bb(\algA)$, valued in the unitary multipliers stabilizing $\bS(\algA,\kg)$, and satisfying
\begin{equation}
\forall f\in\algA_\theta,\quad \forall g\in G\quad:\quad \kL_g(f)=\caE_g\star_\theta f\star_\theta\caE_{g^{-1}}.\label{eq-starexp}
\end{equation}
\end{proposition}
\begin{proof}
The first part of the Proposition is an easy consequence of Theorem \ref{thm-starexp}. Then, the infinitesimal version of Equation \eqref{eq-starexp} is given by $X^*=\frac{i}{\theta} \,[\eta_X,\fois]_{\star_\theta}$ since $X^*$ is the derivative of $\kL_{e^{tX}}$ and $\eta_X$ the one of $\caE_{e^{tX}}=E_{\star_\theta}(\frac{i}{\theta}t\eta_X)$. Integration of this condition provides the result.
\end{proof}

We say that an intertwiner $V:\algA\to \algB$, between two complete HDQ $\algA$ and $\algB$ with the same invariant $G$-symmetry, is {\defin $G$-equivariant} if the map $V_\theta$ is equivariant for the actions of $G$ on $\algA_\theta$ and $\algB_\theta$.
\begin{proposition}
If the HDQs $\algA$ and $\algB$ have the same invariant $G$-symmetry and if $V:\algA\to\algB$ is a $G$-equivariant intertwiner, the bounded and unbounded extensions $\tilde V$ are also $G$-equivariant.
\end{proposition}
\begin{proof}
Indeed, we compute that
\begin{equation*}
\tilde V(\kL_g(T))=V\circ \kL_g\circ T\circ\kL_{g^{-1}}\circ V^{-1}=\kL_g\circ V\circ T\circ V^{-1}\circ \kL_{g^{-1}}=\kL_g(\tilde V(T)),
\end{equation*}
for $g\in G$, $T\in\kM_\bb(\algA)$ or $\kM(\algA)$, due to $G$-invariant symmetry and to the equivariance of $V$.
\end{proof}

\begin{proposition}
Let $\algA$ be a complete HDQ with a $G$-invariant symmetry and $\Omega:\algA\to\caL^2(\ehH)$ a representation of $\algA$. Then, $U_\theta(g):=\Omega_\theta(\caE_g)$ defines a unitary representation $U_\theta:G\to \caU(\ehH_\theta)$ and $\Omega$ is $G$-equivariant:
\begin{equation*}
\forall f\in\algA_\theta,\quad \forall g\in G\quad:\quad \Omega_\theta(\kL_g(f))=U_\theta(g)\Omega_\theta(f)U_\theta(g^{-1}).
\end{equation*}
\end{proposition}
\begin{proof}
This is just a consequence of the definition of a representation (see Definition \ref{def-hdq}) and of Equation \eqref{eq-starexp}.
\end{proof}

\begin{proposition}
Let $\algA$ be a complete HDQ with a $G$-invariant symmetry. Then, 
\begin{itemize}
\item $\bS(\algA_\theta,\kg)$ is the space of smooth vectors of $\algA_\theta$ for the action $G\times G\to\Aut(\algA_\theta)$, $(g,g')\mapsto L_{\caE_g}R^*_{\caE_{g'}}$.
\item $\bH^k(\algA_\theta,\kg)$ is the space of $\caC^k$-vectors of $\algA_\theta$ for the action $G\to\Aut(\algA_\theta)$, $g\mapsto L_{\caE_g}R^*_{\caE_{g}}$. It is independent of the parameter $\theta$.
\end{itemize}
\end{proposition}
\begin{proof}
This result is a consequence of the definition of 
\begin{equation*}
L_{\eta_X}=\frac{\dd}{\dd t}_{|t=0} L_{(\caE_{e^{tX}})},\qquad R_{\eta_X}=\frac{\dd}{\dd t}_{|t=0} R_{(\caE_{e^{tX}})},\qquad L_{\eta_X}-R_{\eta_X}=[\eta_X,\fois]_{\star_\theta}=-i\theta X^*.
\end{equation*}
\end{proof}

\section{Examples of Hilbert deformation quantization}

\subsection{Moyal-Weyl deformation quantization}
\label{subsec-moyal}

We recall the non-formal expression of the Moyal product: $\forall f_1,f_2\in\caD(\gR^{2n})$,
\begin{equation}
(f_1\star_\theta f_2)(x):=\frac{1}{(\pi\theta)^{2n}}\int f_1(y)f_2(z) e^{-\frac{2i}{\theta}(\omega(y,z)+\omega(z,x)+\omega(x,y))}\dd y\dd z,\label{eq-moyalprod}
\end{equation}
where $\omega$ is the standard symplectic form of $\gR^{2n}$, $\theta\in\gR^*$ and $x\in\gR^{2n}$. It satisfies the tracial identity
\begin{equation}
\int (f_1\star_\theta f_2)(x)\dd x=\int f_1(x) f_2(x)\dd x.\label{eq-trident}
\end{equation}
It is then well-known that the star-product $\star_\theta$ extends to $L^2(\gR^{2n})$, and that space endowed with $\star_\theta$, the complex conjugation and the standard scalar product $\langle f_1,f_2\rangle:=\int_{\gR^{2n}} \overline{f_1(x)} f_2(x)\dd x$, is a full (and complete) Hilbert algebra. We then introduce its bounded multiplier algebra $\kM_\bb(L^2(\gR^{2n}))$, and call it the {\defin bounded Moyal multiplier algebra}. It is actually isomorphic to the left bounded multipliers (also used in \cite{Gayral:2003dm}) and it is a von Neumann algebra of type $I_\infty$ with trivial center $\caZ_\bb(L^2(\gR^{2n}))=\gC 1$.

We also recall the form of the Weyl quantization. If $x=(q,p)\in\gR^n\oplus\gR^n=\gR^{2n}$ that is a Lagrangian decomposition with respect to $\omega$, the Weyl map $\Omega_\theta: L^2(\gR^{2n})\to\caL^2(L^2(\gR^n))$ is given by
\begin{equation}
\Omega_\theta(f)\varphi(q_0):=\frac{2^n}{(\pi\theta)^n}\int f(q,p)e^{\frac{4i}{\theta}\omega(q-q_0,p)}\varphi(2q-q_0)\dd q\dd p,\label{eq-weyl}
\end{equation}
for $\varphi\in L^2(\gR^n)$ and $f\in L^2(\gR^{2n})$. See also \cite{Maillard:1986} for study on the Weyl map.

Set $\algA_\theta:=(L^2(\gR^{2n}),\star_\theta)$ for $\theta\neq 0$ and $\algA_0:=L^2(\gR^{2n})\cap L^\infty(\gR^{2n})$ in this section.
\begin{proposition}
The family $\algA:=(\algA_\theta)$ defines a complete Hilbert deformation quantization and the Weyl map \eqref{eq-weyl} is a representation of this HDQ.
\end{proposition}

Let us have a look on some particular multipliers of the Moyal algebra.
\begin{example}
\label{ex-fourier}
The symplectic Fourier transform $\caF=(\caF_L,\caF_R)$ is a unitary multiplier in $\kM_{\bb}(\algA_\theta)$ given by
\begin{equation*}
\caF_L(f)(x)=\frac{1}{(\pi\theta)^n}\int f(y)e^{\frac{2i}{\theta}\omega(x,y)}\dd y,\qquad \caF_R(f)(x)=\frac{1}{(\pi\theta)^n}\int f(y)e^{-\frac{2i}{\theta}\omega(x,y)}\dd y.
\end{equation*}
The associated automorphism is the change of sign: $U_\caF(f)(x):=\caF_L(\caF_R)^*(f)(x)=f(-x)$. We call $\caF=(\caF_L,\caF_R)$ the {\defin Fourier multiplier}. By using the identification of the multipliers with tempered distributions (see Theorem \ref{thm-distribident}), we can see that the Fourier multiplier is associated to the Dirac distribution $(\pi\theta)^n\delta(x)$, as also noticed in \cite{Gayral:2003dm}.
\end{example}
\begin{proof}
First, it is well-known that $\caF_L$ and $\caF_R$ are unitary (normalization has been chosen for this). Then, for $f_1,f_2\in L^2(\gR^{2n})$, we have
\begin{equation*}
f_1\star_\theta \caF_L(f_2) (x)=\frac{1}{(\pi\theta)^n}\int f_1(y)f_2(x-y) e^{-\frac{2i}{\theta}\omega(x,y)}\dd y= \caF_R(f_1)\star_\theta f_2(x),
\end{equation*}
which shows that $(\caF_L,\caF_R)$ is a multiplier.
\end{proof}

\begin{example}
\label{ex-left}
The {\defin translation} by elements of the Abelian group $\gR^{2n}$:
\begin{equation*}
\kL_{x_0}(f)(x)=f(x-x_0)
\end{equation*}
is a group homomorphism $\kL:\gR^{2n}\to \Aut(\algA_\theta)$ valued in the inner automorphisms, due to the $\gR^{2n}$-invariance of the star-product. Indeed,
\begin{equation*}
L_{\kL_{x_0}}(f)(x)=f(x-\frac12 x_0)e^{\frac{i}{\theta}\omega(x_0,x)},\qquad R_{\kL_{x_0}}(f)(x)=f(x+\frac12 x_0)e^{\frac{i}{\theta}\omega(x_0,x)}
\end{equation*}
is a unitary multiplier and $\kL_{x_0}=L_{\kL_{x_0}}(R_{\kL_{x_0}})^*$. Moreover, $\kL$ induces a group homomorphism $\tilde\kL:\gR^{2n}\to\Aut(\kM_\bb(\algA_\theta))$. In the identification of Theorem \ref{thm-distribident}, $\kL_{x_0}$ is associated to the function $x\mapsto e^{\frac{i}{\theta}\omega(x_0,x)}$. It is easy to see the left action on the multipliers: $\tilde \kL_{x_0}(T)(x)=T(x-x_0)$, for any $T\in \kM_\bb(\algA_\theta)$ seen as a subset of $\caS'(\gR^{2n})$.
\end{example}

It turns out that the action of the translation group $\gR^{2n}$, under which the star-product is invariant, is Hamiltonian with respect to the standard symplectic form used for the deformation quantization, and the moment map of the action has the form
\begin{equation*}
\forall x,y\in\gR^{2n}\quad:\quad \eta_x(y)=\omega(x,y)
\end{equation*}
and it is covariant for the star-product: $[\eta_x,\eta_{x'}]_{\star_\theta}=-i\theta\{\eta_x,\eta_{x'}\}=i\theta\omega(x,x')$.
\begin{proposition}
The Lie algebra $\kg\simeq\gR^{2n}$ of the translation group, via the moment map, together with the unit $1$, yields an (invariant) symmetry of the HDQ $\algA$. Moreover, the Schwartz HDQ induced by this symmetry coincides with the Schwartz functions: $\bS(\algA_\theta,\kg)=(\caS(\gR^{2n}),\star_\theta)$.
\end{proposition}
\begin{proof}
Left and right $\star_\theta$-multiplications by linear functions are unbounded operators with domain containing $\caD(\gR^{2n})$, and the covariance of $\star_\theta$ gives the Lie algebra condition on these operators. Moreover, $\caZ_\bb(L^2(\gR^{2n}))=\gC 1$ so it commutes with $\caU(\kg)$. To identify $\algB$, we have to look at its multiplier topology. By Theorem \ref{thm-symhdq}, this topology is given by seminorms $\norm L_SR_T(f)\norm$, where the norm is the $L^2$-norm, $L_S$ is the left $\star_\theta$-multiplication by the polynom $S$ and $R_T$ is the right $\star_\theta$-multiplication by the polynomial $T$. Since, the commutator and the anticommutator have the following expression as unbounded operators
\begin{equation*}
[x_j,f]_{\star_\theta}(x)=i\theta (\omega^{-1}\partial_x)_jf(x),\qquad \{x_j,f\}_{\star_\theta}(x)=2x_jf(x),
\end{equation*}
this topology is generated by the seminorms $\norm x^\alpha\partial^\beta f\norm$ ($\alpha,\beta$ multi-indices), and so it corresponds to the standard topology of the Schwartz space by \cite{Reed:1980}.
\end{proof}
It turns out that $\bS(\algA,\kg)$ induced by the translation symmetry is a continuous deformation quantization. Actually, the situation is much better since M. Rieffel proved in \cite{Rieffel:1993} that $\kM_\bb(\algB)$ is a continuous field of C*-algebras.

\begin{remark}
The framework of symmetries of HDQ is more general than pullback of group actions on the underlying manifold $M$. For example on $M=\gR^2$ endowed with the Moyal-Weyl product, consider the complex countable-dimensional Lie algebra $\kg$ with generators $e^{nx}$ ($x\in\gR^2$ and $n\in\gZ^2$ or $n\in\gN^2$) and satisfying the relations
\begin{equation*}
[e^{nx},e^{mx}]_{\star_\theta}=2i\sin(\frac{\theta}{2}\omega(n,m)) e^{(n+m)x}.
\end{equation*}
This symmetry of $\algA$ also induces a HDQ of Fr\'echet algebras.
\end{remark}

\subsection{Matrix basis and GBV spaces}
\label{subsec-gbv}

In this section, we want to show that the GBV spaces $\bG^{k,l}(\algA_\theta,\kg)$ correspond to the ones introduced in \cite{GraciaBondia:1987kw} with the matrix basis, and that the Sobolev spaces $\bH^k(\algA,\kg)$ correspond to the usual Sobolev spaces $H^k(\gR^{2n})$. To get simpler expressions, let us work in dimension 2 ($n=1$) here, even if the results obtained are valid for arbitrary $n$.

First, we recall the {\defin matrix basis} $(b_{mn})_{m,n\in\gN}$ given in \cite{GraciaBondia:1987kw} by
\begin{equation*}
b_{mn}(x)=2(-1)^m\sqrt{\frac{m!}{n!}} e^{i(n-m)\varphi} \left(\frac{2r^2}{\theta}\right)^{\frac{n-m}{2}} L_m^{n-m}\left(\frac{2r^2}{\theta}\right)\, e^{-\frac{r^2}{\theta}},
\end{equation*}
where we use the polar coordinates $x=(x_1,x_2)=(r\cos(\varphi),r\sin(\varphi))\in\gR^2$ and the Laguerre polynomials $L_m^k$. Such a matrix basis is contained in $\caS(\gR^2)$ and it satisfies
\begin{equation*}
(b_{mn}\star_\theta b_{kl})(x)=\delta_{nk} b_{ml}(x),\qquad \int b_{mn}(x)\dd x=2\pi\theta \delta_{mn},\qquad\overline{b_{mn}(x)}=b_{nm}(x).
\end{equation*}

The above properties allow to show directly the following result.
\begin{proposition}
\label{prop-matrix}
Let $\ell^2(\gN^2)$ be the space of infinite matrices $(f_{mn})_{m,n\in\gN}$ such that $\sum_{m,n}|f_{mn}|^2<\infty$. It is a Hilbert algebra for the usual matrix product, the transpose-conjugation and the scalar product associated to the trace. Moreover, there is a Hilbert algebra isomorphism given by
\begin{equation*}
\ell^2(\gN^2)\to (L^2(\gR^2),\star_\theta),\qquad (f_{mn})\, \mapsto\, f(x)=\sum_{m,n=0}^\infty f_{mn}b_{mn}(x)
\end{equation*}
whose inverse has the form
\begin{equation*}
f\in L^2(\gR^2)\quad\mapsto\quad f_{mn}=\frac{1}{2\pi\theta}\int f(x) b_{nm}(x)\dd x.
\end{equation*}
\end{proposition}
The usual {\defin GBV spaces} are defined \cite{GraciaBondia:1987kw} as
\begin{equation*}
G^{k,l}(\gR^2):=\{f\in L^2(\gR^2),\, \sum_{m,n=0}^\infty m^k n^l|f_{mn}|^2<\infty\}.
\end{equation*}

\begin{proposition}
The GBV spaces $\bG^{k,l}(\algA_\theta,\kg)$ are identical to the usual ones $G^{k,l}(\gR^2)$.
\end{proposition}
\begin{proof}
The translation symmetry $\kg$ contains coordinates $x_1,x_2$. We provide a new basis:
\begin{equation*}
z_1:=\frac{1}{\sqrt{2\theta}}(x_1+ix_2),\qquad z_2:=\overline{z_1}=\frac{1}{\sqrt{2\theta}}(x_1-ix_2).
\end{equation*}
The norm $\norm\fois\norm_{k,l}$ defining the topology of $\bG^{k,l}(\algA_\theta,\kg)$ in Definition \ref{def-sobolev} is expressed in terms of $x_1,x_2\in\kg$, but it can be equivalently reformulated in terms of $z_1,z_2$. Moreover, the isomorphism of Proposition \ref{prop-matrix} can be extended to the polynomials and we have
\begin{equation*}
z_1\,\mapsto\, i\sqrt{m}\,\delta_{m,n+1},\qquad z_2\,\mapsto\, -i\sqrt{m+1}\,\delta_{m+1,n}.
\end{equation*}
Therefore, we can compute the norm in terms of the matrix basis coefficients:
\begin{equation*}
\norm z_{i_1}\star_\theta\dots\star_\theta z_{i_p}\star_\theta f\star_\theta z_{j_q}\star_\theta\dots\star_\theta z_{j_1}\norm^2=\sum_{m,n=0}^\infty (m+\alpha_1)\dots (m+\alpha_p)(n+\beta_1)\dots (n+\beta_q)|f_{mn}|^2,
\end{equation*}
where $\alpha_a,\beta_b$ are real constants depending on the indices $i_1,\dots,i_p$ and $j_1,\dots,j_q$, and satisfying $|\alpha_a|\leq p$, $|\beta_q|\leq q$. With this expression, we see immediatly that $\bG^{k,l}(\algA_\theta,\kg)=G^{k,l}(\gR^2)$.
\end{proof}

Using the expression of the commutator $[x_j,\fois]_{\star_\theta}=i\theta (\omega\partial_x)_j$, we obtain the following result concerning the Sobolev spaces.
\begin{proposition}
The Sobolev spaces $\bH^{k}(\algA_\theta,\kg)$ are identical to the usual ones $H^{k}(\gR^2)$.
\end{proposition}

\subsection{Link with distributions}
\label{subsec-moyaldistr}

Let us try to identify the multiplier space $\kM(\algB)$ of $\algB=\caS(\gR^{2n})$. We will show here that this unbounded notion of multipliers of Hilbert algebra corresponds in this particular case to the well-known notion of multipliers of a Fr\'echet algebra.

Recall that for $(\algB_\theta,\star_\theta)$ be an arbitrary Hilbert subalgebra of $\algA_\theta=(L^2(\gR^{2n}),\star_\theta)$ and a Fr\'echet algebra containing $\caD(M)$ and contained in $\caC^\infty(M)$, we can also define a notion of Fr\'echet multiplier as following. By denoting also $\langle-,-\rangle$ the duality bracket between the distributions $\algB_\theta'$ and the functions $\algA_\theta$, the product $\star_\theta$ satisfying a tracial identity can be extended as
\begin{equation}
\forall T\in\algB_\theta',\ \forall f,h\in\algB_\theta\quad:\quad \langle T\star_\theta f,h\rangle:=\langle T,f\star_\theta h\rangle\text{ and } \langle f\star_\theta T,h\rangle:=\langle T,h\star_\theta f\rangle,\label{eq-moyext}
\end{equation}
which is compatible with the case $T\in\algB_\theta$ (the duality bracket corresponds then to the scalar product). Then, the multiplier space associated to the Fr\'echet algebra $\algB_\theta$ has the form
\begin{equation*}
\caM_{\star_\theta}(\algB_\theta):=\{T\in\algB_\theta',\ f\mapsto T\star_\theta f\text{ and } f\mapsto f\star_\theta T\text{ are continuous from }\algB_\theta \text{ into itself}\}.
\end{equation*}
This space is equipped with the topology associated to the seminorms:
\begin{equation*}
\norm T\norm_{B,j,L}=\sup_{f\in B}\norm T\star_\theta f\norm_{j}\,\text{ and }\, \norm T\norm_{B,j,R}=\sup_{f\in B}\norm f\star_\theta T\norm_{j}
\end{equation*}
where $B$ is a bounded subset of $\algB_\theta$ and $\norm f\norm_{j}$ are the Fr\'echet seminorms of $\algB_\theta$. The product $\star_\theta$ can be extended to $\caM_{\star_\theta}(\algB_\theta)$ by:
\begin{equation*}
\forall S,T\in\caM_{\star_\theta}(\algB_\theta),\ \forall f\in\algB_\theta\quad:\quad \langle S\star_\theta T,f\rangle:=\langle S,T\star_\theta f\rangle=\langle T,f\star_\theta S\rangle.
\end{equation*}
And $(\caM_{\star_\theta}(\algB_\theta),\star_\theta)$ is an associative Hausdorff locally convex algebra, called the {\defin Fr\'echet multiplier algebra}.

\begin{theorem}
\label{thm-distribident}
Let $(\algB_\theta,\star_\theta)$ be a Hilbert subalgebra of $\algA_\theta$ as well as a nuclear Fr\'echet algebra containing $\caD(M)$ and contained in $\caC^\infty(M)$. Then, $\caM_{\star_\theta}(\algB_\theta)\simeq \kM(\algB_\theta)$.
\end{theorem}
\begin{proof}
We will use the following inclusion $\algB_\theta \hookrightarrow \algB_\theta'$ to see first $\kM(\algB_\theta)$ as a subspace of the distributions $\algB_\theta'$. Let us consider an arbitrary $T=(L,R)\in \kM(\algB_\theta)$. Due to Proposition \ref{prop-conttopo}, the maps $L,R:\algB_\theta\to \algB_\theta'$ are continuous for the Fr\'echet topologies. Due to Schwartz kernel's theorem, there exist kernels $K_L,K_R\in \algB'_\theta\hat\otimes\algB'_\theta$ such that $\forall f\in\algB_\theta$, $\forall x\in M$,
\begin{equation*}
L(f)(x)=\int K_L(x,y)f(y)\dd y,\qquad R(f)(x)=\int K_R(x,y)f(y)\dd y,
\end{equation*}
where we use by convenience the symbol integral to mean only the duality bracket for distributions (extending the scalar product of $\algA_\theta$). Implementing the conditions of $L$ to be a left multiplier leads to
\begin{multline*}
\frac{1}{(\pi\theta)^{2n}}\int K_L(x,y-z)f_1(y)\caF_R(f_2)(z) e^{-\frac{2i}{\theta}\omega(y,z)}\dd y\dd z = L(f_1\star_\theta f_2)(x)\\
= L(f_1)\star_\theta f_2(x)= \frac{1}{(\pi\theta)^{2n}}\int K_L(x+z,y)f_1(y)\caF_R(f_2)(z) e^{-\frac{2i}{\theta}\omega(x,z)}\dd y\dd z,
\end{multline*}
for any $f_1,f_2\in \algB_\theta$ (dense in $\algA_\theta$). This is equivalent to the condition
\begin{equation*}
K_L(x,y)=K_L(x-y,0)e^{\frac{2i}{\theta}\omega(x,y)}.
\end{equation*}
In the same way, we have $K_R(x,y)=K_R(x-y,0)e^{-\frac{2i}{\theta}\omega(x,y)}$. We define the following distributions in $\algB'_\theta$:
\begin{equation*}
k_L:=(\pi\theta)^n \caF_R( K_L(\fois,0)),\qquad k_R:=(\pi\theta)^n \caF_L( K_R(\fois,0)).
\end{equation*}
A simple computation gives that $\forall f\in\algB_\theta$,
\begin{equation*}
L(f)(x)=\frac{1}{(\pi\theta)^{2n}}\int k_L(y)f(z) e^{-\frac{2i}{\theta}(\omega(y,z)+\omega(z,x)+\omega(x,y))}\dd y\dd z=: (k_L\star_\theta f)(x),
\end{equation*}
where $\star_\theta$ is defined between $\algB_\theta'$ and $\algB_\theta$ by \eqref{eq-moyext}, i.e. the integral has a distributional meaning. We have also $R(f)(x)=(f\star_\theta k_R)(x)$. Now, the condition $f_1\star_\theta L(f_2)=R(f_1)\star_\theta f_2$ of $(L,R)$ being a multiplier implies that $k_L=k_R$. Since $(L,R)\in\kM(\algB_\theta)$, we have $k_L=k_R\in\caM_{\star_\theta}(\algB_\theta)$ due to its definition. It is then easy to check that
\begin{align*}
& k_{(L_1L_2)}(x)=\frac{1}{(\pi\theta)^{2n}}\int k_{L_1}(y)k_{L_2}(z) e^{-\frac{2i}{\theta}(\omega(y,z)+\omega(z,x)+\omega(x,y))}\dd y\dd z=(k_{L_1}\star_\theta k_{L_2})(x),\\
& k_{(L^*)}(x)=\overline{k_L(x)},
\end{align*}
as distributions, so that the map $(L,R)\in\kM(\algB_\theta)\mapsto k_L=k_R\in\caM_{\star_\theta}(\algB_\theta)$ is a *-algebra morphism. The map $k\in\caM_{\star_\theta}(\algB_\theta)\mapsto (k\star_\theta\fois,\,\fois\star_\theta k)\in\kM(\algB_\theta)$ is obviously an inverse map.
\end{proof}
We can now apply this theorem directly to the algebra $\algB_\theta=(\caS(\gR^{2n}),\star_\theta)$ and we see that its unbounded multiplier algebra $\kM(\algB_\theta)$ corresponds to the Fr\'echet multipliers of $\caS(\gR^{2n})$ (a subspace of the tempered distributions $\caS'(\gR^{2n})$) used in \cite{Maillard:1986,GraciaBondia:1987kw,Gayral:2003dm}.

\begin{remark}
The Weyl map \eqref{eq-weyl} restricted to $\bS(\algA,\kg)$ is extendable. Due to Theorem \ref{thm-symhdq} and since $\caD_{\bS(\algA,\kg)}=\caS(\gR^n)$, we obtain a faithful *-representation
\begin{equation*}
\tilde\Omega_\theta:\caM_{\star_\theta}(\caS(\gR^{2n}))\to \caL^+(\caS(\gR^n))
\end{equation*}
that was already considered in \cite{Maillard:1986} for the Moyal-Weyl quantization.
\end{remark}

\subsection{Infinite-dimensional Clifford algebras}
\label{subsec-clifford}

As an illustration of the unital case, let us consider the well-known hyperfinite type $II_1$ factor, but seen as a limit of deformation quantization. 

Let $V$ be an infinite dimensional separable real vector space with a positive definite scalar product. We consider 
 an orthonormal basis $(\xi_i)_{i\in\gN}$ of $V$. The Clifford algebra $Cl(V)$ consists in the tensor algebra of the complexification of $V$ quotiented by the ideal generated by $\{v\otimes v-\langle v,v\rangle\gone\}$. It is generated by the $\xi_i$ satisfying
 \begin{equation*}
 \xi_i\xi_j+\xi_j\xi_i=2\delta_{ij}\gone.
 \end{equation*}
A basis of $V$ is given by the $\xi_I:=\prod_{i\in I}\xi_i$ where the product is ordered, $\xi_\emptyset=\gone$ and $I$ are subsets of $\gN$.

The following is a consequence of \cite{Plymen:1994}. The algebra $Cl(V)$ has an involution as well as a normalized ($\tau(\gone)=1$) hermitian ($\tau(x^*)=\overline{\tau(x)}$) trace ($\tau(xy)=\tau(yx)$) $\tau:Cl(V)\to\gC$ defined by
\begin{equation*}
x^*:=\sum_I(-1)^{\frac12|I|(|I|-1)} \overline{x_I}\xi_I,\qquad \tau(x)=x_\emptyset,
\end{equation*}
for any $x=\sum_I x_I\xi_I$ with $x_I\in\gC$, $I\subset\gN$ and the sum is finite. The map $(x,y)\mapsto \tau(x^*y)$ is a sesquilinear hermitian positive definite form and we note $\ehH$ the completion of $Cl(V)$ for the norm associated to this sesquilinear form.

Then, $Cl(V)$ acts by left-multiplication on $\ehH$ and we note $Cl[V]$ the weak completion of $Cl(V)$ seen as a subalgebra of $\caL(\ehH)$. $Cl[V]$ is the hyperfinite factor of type $II_1$ and $\tau$ extends to $Cl[V]$ as a finite faithful normal trace. Since $Cl(V)$ coincide with the Moyal deformation quantization \cite{Bieliavsky:2010su,deGoursac:2014kv} of the superspace $\gR^{0|2m}$ if $2m=\dim(V)<+\infty$, and since $\bigcup_{m=1}^\infty Cl(2m)$ is dense in $Cl[V]$, we can see $Cl[V]$ as the limit of the deformation quantization of $\gR^{0|2m}$ when $m\to\infty$. Taking this into account, we have now that $Cl(V)$ is a Hilbert algebra of completion $\ehH$. Due to Proposition \ref{prop-unit}, we have the following result.
\begin{proposition}
The bounded multipliers $\kM_\bb(\algA)$ of $\algA:=Cl(V)$ are left and right multiplication by elements of $Cl[V]$. Moreover, the unbounded multipliers are $\kM(\algA)= Cl(V)$ that is a non-closed O*-algebra.
\end{proposition}

Note that the hyperfactor of type $II_\infty$ can be obtained as the tensor product $\kM_\bb(\gR^{2n})\overline\otimes Cl[V]$, so as the limit of the deformation quantization $\gR^{2n|2m}$ with $m\to\infty$.

\subsection{Deformation quantization of normal j-groups}
\label{subsec-normal}

To show the efficiency of this framework of Hilbert deformation quantization, we look at another deformation quantization that was defined and studied in \cite{Bieliavsky:2010kg}. Let us first describe what are {\defin elementary normal $j$-groups} (see \cite{Pyatetskii-Shapiro:1969}). They are $AN$ Iwasawa factor of the simple Lie groups $SU(1,n)$ for $n\in\gN^*$. Explicitly they are realized as $\gS=\gR\times V\times\gR$, where $(V,\omega)$ is a symplectic vector space of dimension $2n$. With the coordinate system $(a,x,\ell)$ associated to such a realization, the group law of $\gS$ has the form
\begin{equation}
(a,x,\ell)\fois(a',x',\ell')=\Big(a+a',e^{-a'}x+x',e^{-2a'}\ell+\ell'+\frac12e^{-a'}\omega(x,x')\Big)\label{eq-gSlaw}
\end{equation}
and the inverse $(a,x,\ell)^{-1}=(-a,-e^ax,-e^{2a}\ell)$. It turns out that the generic coadjoint orbit of this group is $\gS$-equivariantly diffeomorphic to the Lie group $\gS$ itself and under this identification the KKS symplectic form has the expression $\omega_{\gS}=2\dd a\wedge\dd\ell+\omega$. The coadjoint action (or left-multiplication under the identification) has the associated moment maps
\begin{equation}
\eta_H(a,x,\ell)=2\ell,\qquad \eta_y(a,x,\ell)=e^{-a}\omega(y,x),\qquad \eta_E(a,x,\ell)= e^{-2a},\label{eq-moment}
\end{equation}
with decomposition of the Lie algebra $\ks:=\gR H\oplus V\oplus \gR E$ and exponential map $(a,x,\ell)=e^{aH}e^xe^{\ell E}$.

On such groups, P. Bieliavsky and M. Massar introduced  in \cite{Bieliavsky:2001os} an associative $\gS$-invariant star-product that has the following form
\begin{equation}
(f_1\star_{\theta}f_2)(g)=\frac{1}{(\pi\theta)^{2n+2}}\int\ A_{\gS}(g,g_1,g_2)e^{-\frac{2i}{\theta}S_{\gS}(g,g_1,g_2)}f_1(g_1)f_2(g_2)\dd g_1\dd g_2\label{eq-starprod}
\end{equation}
for $f_1,f_2\in\caD(\gS)$, $g_i:=(a_i,x_i,\ell_i)\in\gS$, the left Haar measure $\dd g:=\dd a\dd x\dd\ell$, and where the amplitude and the phase are
\begin{align*}
A_{\gS}(g,g_1,g_2)=& 4\sqrt{\cosh(2(a_1-a_2))\cosh(2(a_1-a))\cosh(2(a-a_2))}\cosh(a_2-a)^{n}\\
&\cosh(a_1-a)^{n}\cosh(a_1-a_2)^{n},\\
S_{\gS}(g,g_1,g_2)=&-\sinh(2(a_1-a_2))\ell-\sinh(2(a_2-a))\ell_1-\sinh(2(a-a_1))\ell_2\\
&+\cosh(a_1-a)\cosh(a_2-a)\omega(x_1,x_2) +\cosh(a_1-a)\cosh(a_1-a_2)\omega(x_2,x)\\
&+\cosh(a_1-a_2)\cosh(a_2-a)\omega(x,x_1).
\end{align*}

This product, which extends to $L^2(\gS)$, is related to the Moyal-Weyl product \eqref{eq-moyalprod} (but on $\gR^{2n+2}$ and for the symplectic form $\omega_\gS$ instead of $\omega$), that we denote $\star_\theta^0$ in this section, via an intertwining operator $U_\theta$: $f_1\star_{\theta}f_2= U_\theta((U_\theta^{-1}f_1)\star_\theta^0(U_\theta^{-1}f_2))$ for $f_1,f_2\in\caD(\gS)$. These operators have the form
\begin{align}
&U_\theta f(a,x,\ell)=\frac{1}{2\pi}\int\dd t\dd\xi\ \sqrt{\cosh(\frac{\theta t}{2})}\cosh(\frac{\theta t}{4})^n e^{\frac{2i}{\theta}\sinh(\frac{\theta t}{2})\ell-i\xi t}f(a,\cosh(\frac{\theta t}{4})x,\xi),\nonumber\\
&U^{-1}_\theta f(a,x,\ell)=\frac{1}{2\pi}\int\dd t\dd\xi\ \frac{\sqrt{\cosh(\frac{\theta t}{2})}}{\cosh(\frac{\theta t}{4})^n}e^{-\frac{2i}{\theta}\sinh(\frac{\theta t}{2})\xi+it\ell}f(a,\cosh(\frac{\theta t}{4})^{-1}x,\xi).\label{eq-intertw}
\end{align}
Finally the product \eqref{eq-starprod} is associated to the following quantization map \cite{Bieliavsky:2010kg}
\begin{multline}
\Omega_{\theta}(f)\varphi(a_0,v_0):=\frac{2}{(\pi\theta)^{n+1}}\int_{\gS} f(a,v,w,\ell)\sqrt{\cosh(2(a-a_0))}\cosh(a-a_0)^n\\
 e^{\frac{2i}{\theta}\Big(\sinh(2(a-a_0))\ell+\omega(\cosh(a-a_0)v-v_0,\cosh(a-a_0)w)\Big)} \varphi(2a-a_0,2\cosh(a-a_0)v-v_0)\dd a\dd v\dd w\dd\ell,\label{eq-qumapelem}
\end{multline}
for $f\in\caD(\gS)$, $\varphi\in L^2(Q)$, $Q:=\gR H\times V_0$, $V_0$ a Lagrangian subspace of $V$, $(v,w)\in V$ in the integral, and $(a_0,v_0)\in Q$.

In \cite{Bieliavsky:2010kg}, a modified Schwartz space was introduced on $\gS$. Let us recall its definition. The left-invariant vector fields of $\gS$ are given by
\begin{equation*}
\tilde H=\partial_a-x\partial_x-2\ell\partial_\ell,\quad \tilde y=y\partial_x+\frac12\omega(x,y)\partial_\ell,\quad \tilde E=\partial_\ell,
\end{equation*}
with $H$, $y$ and $E$ generators of the Lie algebra $\ks$. The maps $\alpha$ are
\begin{equation*}
\alpha_H(g)=\ell,\qquad \alpha_y(g)=\cosh(a)\omega(y,x),\qquad \alpha_E(g)=\sinh(2a),
\end{equation*}
for any $g=(a,x,\ell)\in\gS$. This leads to the following definition. The {\defin modified Schwartz space} of $\gS$ is defined as
\begin{equation}
\caS(\gS)=\{f\in C^\infty(\gS)\quad \forall j\in\gN^{2n+2},\ \forall P\in\caU(\ks)\quad
\norm f\norm_{j,P}:=\sup_{g\in\gS}\Big| \alpha^j(g) \tilde P f(g)\Big|<\infty\}\label{eq-modsch}
\end{equation}
where $\alpha^j:=\alpha_H^{j_1}\alpha_{e_1}^{j_2}\dots\alpha_{e_{2n}}^{j_{2n+1}} \alpha_E^{j_{2n+2}}$. It is  a Fr\'echet nuclear algebra endowed with the seminorms $(\norm f\norm_{j,P})$ and the star-product \eqref{eq-starprod}.
\medskip

Let us see how the above setting fits in the formalism of HDQ. Denote $\algA_\theta:=L^2(\gS)$ for $\theta\neq 0$ and $\algA_0:=L^2(\gS)\cap L^\infty(\gS)$.
\begin{proposition}
\label{prop-hdqelem}
The family $\algA:=(\algA_\theta)$ defines a Hilbert deformation quantization and the operator \eqref{eq-intertw} is an intertwiner (in the sense of Definition \ref{def-hdq}) between the Moyal-Weyl HDQ $(L^2(\gR^{2n+2}),\star_\theta^0)$ and $\algA$. Moreover, the quantization map \eqref{eq-qumapelem} is related to the Weyl map \eqref{eq-weyl}, denoted now by $\Omega_\theta^0$, in the following way: $\Omega_\theta=\Omega_\theta^0\circ U_\theta^{-1}$, and it is a representation of $\algA$.
\end{proposition}
\begin{proof}
First, we know from \cite{Bieliavsky:2001os} that $U_\theta$ is an algebra homomorphism. It is easy to see that it is also compatible with complex conjugation and scalar products:
\begin{multline*}
\int |U_\theta f(a,x,\ell)|^2\dd a\dd x\dd\ell\\
=\frac{1}{2\pi}\int \cosh(\frac{\theta t}{4})^{2n} e^{i(\xi-\xi')t}\overline{f(a,\cosh(\frac{\theta t}{4})x,\xi)} f(a,\cosh(\frac{\theta t}{4})x,\xi')\dd a\dd x\dd t\dd\xi\dd\xi'\\
=\int |f(a,x,\xi)|^2\dd a\dd x\dd\xi,
\end{multline*}
first by integrating on $\ell$, then by changing the variable $x$ and integrating over $t$. So $U_\theta$ is an isometric isomorphism from $(L^2(\gR^{2n+2}),\star_\theta^0)$ and $(L^2(\gS),\star_\theta)$, which shows that $(L^2(\gS),\star_\theta)$ is a complete Hilbert algebra. Actually, it was already directly proved in \cite{Bieliavsky:2010kg} that this was a Hilbert algebra. Then, a direct computation shows the relation $\Omega_\theta=\Omega_\theta^0\circ U_\theta^{-1}$.
\end{proof}
\begin{corollary}
$\kM_\bb(\algA)$ is a continuous field of C*-algebras.
\end{corollary}
\begin{proof}
This is a direct consequence of the fact that $\kM_\bb(L^2(\gR^{2n+2}),\star_\theta)$ is a continuous field of C*-algebras \cite{Rieffel:1993}, of the fact that $U:L^2(\gR^{2n+2})\to L^2(\gS)$ defined by \eqref{eq-intertw} is an intertwiner and of Proposition \ref{prop-contfield}.
\end{proof}

\begin{example}
\label{ex-fourierelem}
Let us define here a particular multiplier of $\algA_\theta$ that gives a new {\defin Fourier transformation}. The symplectic Fourier transformation is a multiplier of the Moyal-Weyl HDQ associated to the distribution $(\pi\theta)^{n}\delta(x)$ (see Example \ref{ex-fourier}). Let us push such a Dirac distribution by the intertwiner $U_\theta$:
\begin{equation*}
\delta_\gS(a,x,\ell)=(\pi\theta)^{n+1}U_\theta(\delta)(a,x,\ell)=(\pi\theta)^n\delta(a)\delta(x)\int (1+t^2)^{-\frac14}c(t)^{-n}e^{\frac{2i}{\theta}t\ell}\dd t
\end{equation*}
Then, this Fourier transform $\caF_\gS=(\caF_{\gS,L},\caF_{\gS,R})$ is the unitary multiplier in $\kM_{\bb}(\algA_\theta)$ associated to $\delta_\gS$, namely
\begin{align*}
\caF_{\gS,L}(f)(a,x,\ell)=\frac{4}{(\pi\theta)^{n+1}}\int &\sqrt{\cosh(2a)\cosh(2a')}  \cosh(a)^n\cosh(a')^n f(a',x',\ell')\\
&e^{-\frac{2i}{\theta}(\sinh(2a')\ell-\sinh(2a)\ell'+\cosh(a)\cosh(a')\omega(x',x))}\dd a'\dd x'\dd\ell',\\
\caF_{\gS,R}(f)(a,x,\ell)=\frac{4}{(\pi\theta)^{n+1}}\int & \sqrt{\cosh(2a)\cosh(2a')}  \cosh(a)^n\cosh(a')^n f(a',x',\ell')\\
& e^{+\frac{2i}{\theta}(\sinh(2a')\ell-\sinh(2a)\ell'+\cosh(a)\cosh(a')\omega(x',x))}\dd a'\dd x'\dd\ell'.
\end{align*}
We have therefore defined two unitary transformations on $\algA_\theta$. The associated automorphism is also the change of sign: $U_{\caF_\gS}(f)(a,x,\ell):=\caF_{\gS,L}(\caF_{\gS,R})^*(f)(a,x,\ell)=f(-a,-x,-\ell)$.
\end{example}

\subsection{Symmetry for normal j-groups}
\label{subsec-normalsymm}

Let us now identify the space $\caS(\gS)$ in the framework of HDQ.
\begin{theorem}
\label{thm-symelem}
The subspace of $\caC^\infty(\gS)$ generated by the functions $\ell$, constants, $e^{-a}x_j$, $e^{a}x_j$ ($j\in\{1,\dots,2n\}$), $e^{-2a}$ and $e^{2a}$ in the coordinate system $(a,x,\ell)$ of $\gS$ forms a symmetry $\kg$ of the HDQ $\algA$ of Proposition \ref{prop-hdqelem}. The Schwartz HDQ $\algB:=\bS(\algA,\kg)$ induced by this symmetry corresponds to the modified Schwartz space $\caS(\gS)$ defined in \eqref{eq-modsch}.
\end{theorem}
It turns out that $\kg$ is a representation of the Lie algebra of the transvection group of the bounded symmetric domain associated (and isomorphic) to $\gS$, so that we can call this symmetry $\kg$ the {\defin transvection symmetry}. The action of this transvection group is actually an invariant symmetry of the HDQ.
\begin{proof}
First, the above generators form a symmetry of $\algA$ with Lie relations
\begin{align*}
&[\ell,e^{-2a}]_{\star_\theta}=-i\theta e^{-2a},\qquad [\ell,e^{2a}]_{\star_\theta}=i\theta e^{2a},\qquad [e^{\eps a}\omega(y,x),e^{\eps' a}\omega(y',x)]_{\star_\theta}=-i\theta e^{(\eps+\eps')a}\omega(y,y'),\\
&[\ell,e^{-a}\omega(y,x)]_{\star_\theta}=-\frac{i\theta}{2}e^{-a}\omega(y,x),\qquad [\ell,e^{a}\omega(y,x)]_{\star_\theta}=\frac{i\theta}{2}e^{a}\omega(y,x),
\end{align*}
for $\eps,\eps'=\pm 1$. Let us show that the multiplier topology of $\algB_\theta$, i.e. the one generated by the seminorms $\norm L_SR_T f\norm$, for $f\in\algB_\theta$, $S,T\in\caU(\kg)$, corresponds to the topology of $\caS(\gS)$, i.e. the one generated by the seminorms $\norm \sinh(2a)^{k_1} x^\alpha\ell^{k_2}\partial_a^{k_3}\partial_x^\beta\partial_\ell^{k_4} f\norm_\infty$, for $k_i\in\gN$ and $\alpha,\beta\in\gN^{2n}$. It turns out that the part of the multiplier topology corresponding to the generator $\ell$ coincides obviously with the part of the topology of $\caS(\gS)$ corresponding to the operators $\ell^{k_2}\partial_a^{k_3}$, as in the case of the Moyal product. Let us compare the part of the multiplier topology corresponding to the generators $e^{-2a},e^{2a}$ with the part of the topology of $\caS(\gS)$ corresponding to the operators $\sinh(2a)^{k_1}\partial_\ell^{k_4}$.

A direct computation using the explicit expression \eqref{eq-starprod} shows that
\begin{align*}
&[e^{-2ka},f]_{\star_\theta}(a,x,\ell)=-\frac{2}{\pi\theta}e^{-2ka}\int_{\gR^2}\sinh(k\fois \text{arcsinh}(u))e^{\frac{2i}{\theta}u(\ell-v)}f(a,x,v)\dd u\dd v,\\
& \{e^{-2ka},f\}_{\star_\theta}(a,x,\ell)=\frac{2}{\pi\theta}e^{-2ka}\int_{\gR^2}\cosh(k\fois \text{arcsinh}(u))e^{\frac{2i}{\theta}u(\ell-v)}f(a,x,v)\dd u\dd v,
\end{align*}
where the anticommutator is denoted as $\{f,g\}_{\star_\theta}=f\star_\theta g+g\star_\theta f$.
\begin{itemize}
\item If $k$ is even, $\cosh(k\fois \text{arcsinh}(u))=P_k(u)\eps_0(u)$, with $P_k$ a real polynomial of degree $k$ and $\eps_0(u)=1$.
\item If $k$ is odd, $\cosh(k\fois \text{arcsinh}(u))=P_k(u)\eps_1(u)$, with $P_k$ a real polynomial of degree $k-1$ and $\eps_1(u)=\sqrt{1+u^2}$.
\item If $k$ is even, $\sinh(k\fois \text{arcsinh}(u))=Q_k(u)\eps_1(u)$, with $Q_k$ a real polynomial of degree $k-1$ and $\eps_1(u)=\sqrt{1+u^2}$.
\item If $k$ is odd, $\sinh(k\fois \text{arcsinh}(u))=Q_k(u)\eps_0(u)$, with $Q_k$ a real polynomial of degree $k$ and $\eps_0(u)=1$.
\end{itemize}
In particular, $P_1(u)=1$, $Q_1(u)=u$, $P_2(u)=1+2u^2$, $Q_2(u)=2u$, $P_3(u)=1+4u^2$, $Q_3(u)=3u+4u^3$,... We can now compute the $L^2$-norms of these quantities. For example,
\begin{multline*}
\norm [e^{-2ka},f]_{\star_\theta}\norm^2=\frac{4}{\pi\theta}\int Q_k(u)^2\eps_{k+1}(u)^2 e^{-4ka}e^{-\frac{2i}{\theta}u(v'-v)}\overline{f(a,x,v)}f(a,x,v'),\dd a\dd u\dd v\dd v'\dd x\\
=\frac{4}{\pi\theta}\int  e^{-4ka}e^{-\frac{2i}{\theta}u(v'-v)}\overline{f(a,x,v)} (Q_k\eps_{k+1})^2(\frac{-i\theta}{2}\partial_{v'}) f(a,x,v'),\dd a\dd u\dd v\dd v'\dd x
\end{multline*}
by using integration by parts since $(Q_k\eps_{k+1})^2$ is always a polynomial. We get
\begin{itemize}
\item If $k$ is even, $\norm \{e^{-2ka},f\}_{\star_\theta}\norm^2=4\norm e^{-2ka} P_k(\frac{-i\theta}{2}\partial_\ell)f\norm^2$.
\item If $k$ is odd, $\norm \{e^{-2ka},f\}_{\star_\theta}\norm^2=4\norm e^{-2ka} P_k(\frac{-i\theta}{2}\partial_\ell)f\norm^2+\theta^2\norm e^{-2ka} P_k(\frac{-i\theta}{2}\partial_\ell)\partial_\ell f\norm^2$.
\item If $k$ is even, $\norm [e^{-2ka},f]_{\star_\theta}\norm^2=4\norm e^{-2ka} Q_k(\frac{-i\theta}{2}\partial_\ell)f\norm^2+\theta^2\norm e^{-2ka} Q_k(\frac{-i\theta}{2}\partial_\ell)\partial_\ell f\norm^2$.
\item If $k$ is odd, $\norm [e^{-2ka},f]_{\star_\theta}\norm^2=4\norm e^{-2ka} Q_k(\frac{-i\theta}{2}\partial_\ell)f\norm^2$.
\end{itemize}
Since the star-mutliplication of powers of $e^{-2a}$ and $e^{2a}$ is just the pointwise multiplication, we see that the $a$-part of the multiplier topology of $\algB_\theta$ is generated by the seminorms $\norm \{e^{-2ka},f\}_{\star_\theta}\norm$ and $\norm [e^{-2ka},f]_{\star_\theta}\norm$. And these are equivalent to the seminorms $\norm e^{-2k_1 a}\partial_\ell^{k_2}f\norm$ (for the use of the $L^2$-norm), and also to the seminorms $\norm \sinh(2a)^{k_1}\partial_\ell^{k_2}f\norm$ (by using the inequality $\cosh(2a)\leq 1+\sinh(2a)^2$).

What remains to be done for this part concerning the generators $e^{\pm 2a}$ is the equivalence between these seminorms and the seminorms using the same operators but with the $L^\infty$-norm. We adapt the standard argument of \cite{Reed:1980} and we concentrate on the variable $a$ and look at the low order case. Since $(1+\sinh(2a)^2)^{-1}$ is in $L^2(\gR)$, then
\begin{equation*}
\norm f\norm_2\leq \norm (1+\sinh(2a)^2)^{-1}\norm_2 \norm (1+\sinh(2a)^2)f\norm_\infty\leq C(\norm f\norm_\infty+\norm \sinh(2a)^2f\norm_\infty).
\end{equation*}
For the other sense of the equivalence, we use $f(a)=\int_{-\infty}^a \partial_{a'}f(a')\dd a'$ and we obtain
\begin{equation*}
\norm f\norm_\infty\leq \norm\partial_a f\norm_1\leq \norm (1+\sinh(2a)^2)^{-1}\norm_2 \norm (1+\sinh(2a)^2)\partial_a f\norm_2
\end{equation*}
by using also the Cauchy-Schwartz inequality.

Now let us compare the part of the multiplier topology corresponding to the generator $x$ with the part of the topology of $\caS(\gS)$ corresponding to the operators $x^\alpha\partial_x^\beta$. The philosophy is the same as before so let us do it for degree 1. First, a computation gives the following expressions (for $\eps=\pm 1$):
\begin{align*}
&[e^{\eps a}\omega(y,x),f]_{\star_\theta}=i\theta e^{\eps a}(y\partial_x -\frac{\eps}{2}\omega(y,x)\partial_\ell)f(a,x,\ell)\\
&\{e^{\eps a}\omega(y,x),f\}_{\star_\theta}=\frac{e^{\eps a}}{\pi\theta}\int e^{-\frac{2i}{\theta}t(s-\ell)}\left(2c(t)^2\omega(y,x)+i\theta\eps \frac{s(t)}{c(t)}y\partial_x\right)f(a,x,s)\dd t\dd s
\end{align*}
with $c(t):=\cosh(\frac12\text{arcsinh}(t))=\left(\frac{1+\sqrt{1+t^2}}{2}\right)^{\frac12}$ and $s(t):=\sinh(\frac12\text{arcsinh}(t))$.

We obtain directly the inequality $\norm [e^{\eps a}\omega(y,x),f]_{\star_\theta}\norm\leq \theta\norm \sinh(2a)y\partial_x f\norm+ \frac{\theta}{2}\norm \sinh(2a)\omega(y,x)\partial_\ell f\norm$ for the commutator. For the anticommutator, notice that $\caF^{-1}T\caF$ is a positive operator as soon as $T$ is a positive operator, where $\caF$ denotes here the usual Fourier transform. Then we use the inequality $(a+b)^2\leq 2(a^2+b^2)$ for the commuting selfadjoint operators $\omega(y,x)$ and $iy\partial_x$:
\begin{equation*}
\Big(2c(t)^2\omega(y,x)-i\theta\frac{s(t)}{c(t)}y\partial_x\Big)^2\leq 2\Big(4c(t)^4\omega(y,x)^2-\theta^2\frac{s(t)^2}{c(t)^2}(y\partial_x)^2\Big)
\end{equation*}
and we obtain for the anticommutator
\begin{equation*}
\norm \{e^{\eps a}\omega(y,x),f\}_{\star_\theta}\norm^2\leq 8\norm e^{\eps a}\omega(y,x)f\norm^2+2\theta^2\norm e^{\eps a}\omega(y,x)\partial_\ell f\norm^2+\frac{\theta^4}{8}\norm e^{\eps a}y\partial_x\partial_\ell f\norm^2.
\end{equation*}
by using integrations by parts and transforming $t$ in $\frac{i\theta}{2}\partial_s$. This means that the multiplier topology is controlled by the one of $\caS(\gS)$ (we have also the equivalence between $L^2$-norms and $L^\infty$-seminorms for these operators). For the converse sense, concerning the operators $x^\alpha\partial_x^\beta$, we have $y\partial_x f=\frac{i}{\theta}(e^a[e^{-a}\omega(y,x),f]_{\star_\theta}- e^{-a}[e^{a}\omega(y,x),f]_{\star_\theta})$ so that
\begin{equation*}
\norm y\partial_x f\norm\leq \frac{1}{\theta}\norm \cosh(2a)[e^{-a}\omega(y,x),f]_{\star_\theta}\norm +\frac{1}{\theta}\norm \cosh(2a)[e^{a}\omega(y,x),f]_{\star_\theta}\norm
\end{equation*}
and we know already that $\norm\cosh(2a)f\norm$ is controlled by the multiplier topology. Since we have
\begin{multline*}
\Big(2c(t)^2\omega(y,x)-i\theta\frac{s(t)}{c(t)}y\partial_x\Big)^2+\Big(2c(t)^2\omega(y,x)+i\theta\frac{s(t)}{c(t)}y\partial_x\Big)^2\\
=2\left(4c(t)^4\omega(y,x)^2-\theta^2\frac{s(t)^2}{c(t)^2}(y\partial_x)^2\right)\geq 8\omega(y,x)^2,
\end{multline*}
by using the above computations and arguments, we find that
\begin{equation*}
\norm \omega(y,x)f\norm^2\leq \frac{1}{8}\norm e^a \{e^{- a}\omega(y,x),f\}_{\star_\theta}\norm^2+\frac{1}{8}\norm e^{-a} \{e^{ a}\omega(y,x),f\}_{\star_\theta}\norm^2.
\end{equation*}
So, the multiplier topology of $\algB_\theta$ is equivalent to the Fr\'echet topology of $\caS(\gS)$ and we have $\algB_\theta=\caS(\gS)$.
\end{proof}

\begin{remark}
Theorem \ref{thm-symelem} extends to the framework of normal $j$-groups, or equivalently to K\"ahlerian Lie groups with negative curvature. Such {\defin normal $j$-groups} \cite{Pyatetskii-Shapiro:1969} are actually semidirect products of elementary normal $j$-groups
\begin{equation*}
G=(\dots (\gS_1\ltimes_{\rho_1}\gS_2)\ltimes_{\rho_2}\dots\gS_{N-1})\ltimes_{\rho_{N-1}}\gS_N
\end{equation*}
where $\rho_i:\gS_i\to\Aut(\gS_{i+1})$ are some (symplectic) actions. In \cite{Bieliavsky:2010kg}, it has been proved that the star-products for $G$ were tensor products of the elementary factors $\gS_i$ and that the modified Schwartz space was also a tensor product: $\caS(G):=\caS(\gS_1)\hat\otimes\dots\hat\otimes\caS(\gS_N)$. Then, by taking the generators given in Theorem \ref{thm-symelem} for each factor $\gS_i$, we obtain a symmetry of the HDQ $(L^2(G),\star_\theta)$ given by the tensor product of each $(L^2(\gS_i),\star_\theta)$, and the Schwartz induced HDQ corresponds to the tensor product $\caS(G)$. In the same way, results below concerning elementary normal $j$-groups $\gS$ can be extended to normal $j$-groups.
\end{remark}

\begin{remark}
Even if the kernel of the star-product \eqref{eq-starprod} is different from the one of the Moyal-Weyl product \eqref{eq-moyalprod}, a slight adaptation of the proof of Theorem \ref{thm-distribident} leads to the same kind of result, namely that the unbounded multiplier space $\kM(\algB_\theta)$, which is a complete nuclear locally convex Hausdorff *-algebra by Proposition \ref{prop-lchalg}, identifies with the Fr\'echet multiplier space $\caM_{\star_\theta}(\caS(\gS))$ defined and used for the star-exponential in \cite{Bieliavsky:2013sk}.
\end{remark}

Let us call the generators of $\kg$ by $H$, $y$, $y'$, $E$, $E'$ (with $y,y'\in V$) such that the moment maps are
\begin{equation*}
\eta_H=2 \ell,\quad \eta_y=e^{-a}\omega(y,x),\quad \eta_{y'}=e^{a}\omega(y',x),\quad \eta_E= e^{-2a},\quad \eta_{E'}=e^{2a}.
\end{equation*}
\begin{corollary}
By Theorem \ref{thm-starexp}, for any $T\in\kg$, the star-exponential $E_{\star_\theta}(\frac{i}{\theta}T)$ of the transvection symmetry $\kg$ belongs to $\kMstab(\algB_\theta)$, so it is a unitary multiplier stabilizing $\algB_\theta\simeq \caS(\gS)$. We can find its explicit expression by using \cite{Bieliavsky:2013sk}. Namely, any arbitrary $T$ in $\kg$ writes
\begin{equation*}
T(a,x,\ell)=\alpha \eta_H+\eta_y+\eta_{y'}+\beta \eta_E+\beta' \eta_{E'},
\end{equation*}
with $\alpha,\beta,\beta'\in\gR$ and $y,y'\in V$. So, the star-exponential $E_{\star_\theta}(\frac{i}{\theta}T)$ is the left and right $\star_\theta$-multiplication by the function
\begin{equation*}
(a,x,\ell)\mapsto \sqrt{\cosh(\alpha)}\cosh(\frac{\alpha}{2})^n\ e^{\frac{i}{\theta}\sinh(\alpha)\Big(2\ell+\frac{\beta}{\alpha}e^{-2a}\ -\frac{\beta'}{\alpha}e^{2a}+\frac{e^{-a}}{\alpha}\omega(y,x)-\frac{e^{a}}{\alpha}\omega(y',x)\Big)}.
\end{equation*}
\end{corollary}

\begin{corollary}
The representation \eqref{eq-qumapelem} restricted to $\algB$ is extendable, and it extends as a faithful *-representation $\Omega:\kM(\algB)\to\caL^+(\caD_\algB)$ of the multipliers of $\algB$, with
\begin{equation*}
\caD_\algB:=\{\varphi\in L^2(\gR^{n+1}),\, \forall k_i\in\gN,\,\forall \alpha_i\in\gN^n,\, \sup_{a,v} |\sinh(2a)^{k_1} v^{\alpha_1}\partial_a^{k_2}\partial_v^{\alpha_2}\varphi(a,v)|<\infty\}.
\end{equation*}
\end{corollary}
\begin{proof}
Due to Theorem \ref{thm-symhdq}, it suffices to compute the image of $\Omega$ on generators of $\caU(\kg)$ to identify $\caD_\algB$. By using the explicit expression \eqref{eq-qumapelem}, we have
\begin{align*}
&\Omega(\ell^k)\varphi(a_0,\ell_0)=\left(\frac{i\theta}{4}\right)^k P_k(\partial_{a_0})\varphi(a_0,v_0),\qquad \Omega(\sinh(2a)^k)\varphi(a_0,v_0)=\sinh(2a_0)^k\varphi(a_0,v_0),\\
&\Omega(v^\alpha)\varphi(a_0,v_0)=v_0^\alpha\varphi(a_0,v_0),\qquad \Omega(w^\alpha)\varphi(a_0,v_0)=\Big(-\frac{i\theta}{2}\omega^{-1}\partial_{v_0}\Big)^\alpha \varphi(a_0,v_0),
\end{align*}
where $P_k$ is a polynomial of degree $k$. So the topology of $\caD_\algB$ is the exactly the one of the Corollary.
\end{proof}

We can now give an explicit expression for the new Sobolev spaces associated to this symmetry. 
\begin{proposition}
The Sobolev spaces associated to the transvection symmetry are
\begin{equation*}
\bH^k(\algA,\kg)=\{f\in L^2(\gS),\, \forall X\in\caU_k(\kg),\, \norm X^*f\norm<\infty\},
\end{equation*}
where $\caU_k(\kg)$ denotes the filtered enveloping algebra with generators of degree less or equal than $k$, fundamental vector fields are extended to elements of $\caU(\kg)$ by $(X_1\dots X_p)^*=X_1^*\dots X_p^*$, and the fundamental vectors have the expression
\begin{align*}
&H^*=-\partial_a,\qquad y^*=-e^{-a}\partial_y+\frac12e^{-a}\omega(x,y)\partial_\ell,\qquad E^*=-e^{-2a}\partial_\ell,\\
&(y')^*=-e^{a}\partial_{y'}-\frac12e^{a}\omega(x,y')\partial_\ell,\qquad (E')^*=-e^{2a}\partial_\ell.
\end{align*}
\end{proposition}

\begin{proposition}
Since $\star_\theta$ is strongly-invariant under the left-action of $\gS$ on itself, the moments maps \eqref{eq-moment} form another symmetry of the HDQ $\algA$. However, the Schwartz induced HDQ does not correspond to the one of Theorem \ref{thm-symelem}.
\end{proposition}
\begin{proof}
The moment maps satisfy the relations:
\begin{equation*}
[\eta_H,\eta_y]_{\star_\theta}=-i\theta\eta_y,\qquad [\eta_H,\eta_E]_{\star_\theta}=-2i\theta\eta_E,\qquad [\eta_y,\eta_{y'}]_{\star_\theta}=-i\theta\omega(y,y')\eta_E,
\end{equation*}
so they form a symmetry of the HDQ $\algA$. Since the moment maps do not include the function $e^{2a}$, we see that the induced family of seminorms will correspond to a very different behavior from the ones of Theorem \ref{thm-symelem} for $a\to\infty$.
\end{proof}

\vskip 0.5 cm
\noindent {\bf Acknowledgements}: The author thanks Jean-Pierre Antoine for introducing him to the subject of O*-algebras, and Pierre Bieliavsky, Victor Gayral, Bruno Iochum and Yoshiaki Maeda for interesting discussions on the part concerning deformation quantizations.

\bibliographystyle{utcaps}
\bibliography{biblio-these,biblio-perso,biblio-recents}

\end{document}